\documentclass[11pt]{amsart}
\usepackage{amsmath}
\usepackage{amssymb}
\usepackage{epsfig}
\usepackage{graphics}
\usepackage{amsthm}
\usepackage{bm}
\usepackage{url}

\usepackage[dvipdfmx]{hyperref}
\usepackage{color}

\numberwithin{equation}{section}

\newtheorem{theorem}{Theorem}[section]
\newtheorem{lemma}[theorem]{Lemma}

\newtheorem{proposition}[theorem]{Proposition}
\newtheorem{Definition}[theorem]{Definition}

\newenvironment{definition}{\begin{Definition}\rm}{\end{Definition}}
\newtheorem{Rem}[theorem]{Remark}
\newenvironment{remark}{\begin{Rem}\rm}{\end{Rem}}

\numberwithin{figure}{section}

\usepackage{graphicx}
\ifx\pdfoutput\undefined
\else
\fi
\graphicspath{{images/}}
\newcommand{\dps}{\displaystyle}
\newcommand{\ds}{\displaystyle}

\newcommand\R{\mathbb{R}}

\newcommand\be{\begin{equation}}
\newcommand\ee{\end{equation}}
\newcommand\bea{\begin{eqnarray}}
\newcommand\eea{\end{eqnarray}}
\newcommand\beaa{\begin{eqnarray*}}
\newcommand\eeaa{\end{eqnarray*}}
\newcommand\beba{\begin{equation}\left\{\begin{array}{rcl}}
\newcommand\eeba{\end{array}\right.\end{equation}}
\newcommand\bebaa{\begin{equation*}\left\{\begin{array}{rcl}}
\newcommand\eebaa{\end{array}\right.\end{equation*}}
\newcommand\beca{\begin{equation}\left\{\begin{array}{rcll}}
\newcommand\eeca{\end{array}\right.\end{equation}}
\newcommand\becaa{\begin{equation*}\left\{\begin{array}{rcll}}
\newcommand\eecaa{\end{array}\right.\end{equation*}}
\newcommand\bela{\begin{equation}\left\{\begin{array}{l}}
\newcommand\eela{\end{array}\right.\end{equation}}
\newcommand\belaa{\begin{equation*}\left\{\begin{array}{l}}
\newcommand\eelaa{\end{array}\right.\end{equation*}}
\newcommand\bR{{\Bbb R}}

\newcommand\p{{\partial}}
\newcommand{\ep}{\varepsilon}

\definecolor{purple}{rgb}{.65,0,.65}

\begin{document}
\title[Global existence and uniqueness of solutions for reaction-interface systems]{Global existence and uniqueness of solutions for one-dimensional reaction-interface systems}

\author{Yan-Yu Chen}
\address{Department of Mathematics, National Taiwan University,
Taipei 10617, Taiwan}
\email{chenyanyu24@gmail.com}

\author{Hirokazu Ninomiya}
\address{School of Interdisciplinary Mathematical Sciences, Meiji University
4-21-1 Nakano, Nakano-ku, Tokyo 164-8525, Japan}
\email{hirokazu.ninomiya@gmail.com}

\author{Chang-Hong Wu}
\address{Department of Applied Mathematics,
 National Yang Ming Chiao Tung University, Hsinchu 300, Taiwan}
\email{changhong@math.nctu.edu.tw}

\begin{abstract}
In this paper, we provide a mathematical framework in studying the wave propagation with the annihilation phenomenon
in excitable media. We deal with the existence and uniqueness of solutions to a one-dimensional free boundary problem (called a reaction--interface system) arising from the singular limit of a FitzHugh--Nagumo type reaction--diffusion system.
Because of the presence of the annihilation, interfaces may intersect each other.
We introduce the notion of weak solutions to study the continuation of solutions beyond the annihilation time.
Under suitable conditions, we show that the free boundary problem is well-posed.
\end{abstract}

\date{\today}


\maketitle \setlength{\baselineskip}{15pt}
%%%%%%%%%%%%%%%%%%%%%%%%%%%%%%%%%%%%%%%%%%%%%%%%%%%%%%%%%%%%%%%%%
\section{Introduction}
\setcounter{equation}{0}
A variety of wave patterns can be triggered in excitable media such as traveling front, pulse waves,
periodic wave trains, rotating spirals, and so on. Reaction-diffusion systems have been successfully modeling these spatio-temporal patterns. A wide class of spatio-temporal patterns has been discussed in, for example,
\cite{ASM2003,BJG,BJS,HM,KS,KM,Meron92,MSCS2,MZ,Pismen,Tyson88,Winfree1972}
and the references cited therein. A fundamental phenomenon observed in the experiments is that
chemical waves propagating at a roughly constant speed may
collide with each other, and they annihilate one another. Though it might be a relatively simple
phenomenon in excitable media, it is still challenging to be proved theoretically.
To better understand this issue, considering interface problems is one of possible
ways to model this phenomenon (e.g., \cite{fife1984}). In general, solutions of interface problems produce transition layers (called interfaces), which separate the domain into different phase regions.
Interface problems have been deduced from nonlinear reaction-diffusion equations such as
Allen--Cahn type equation, Belousov--Zhabotinsky (BZ) systems, competition-diffusion systems or FitzHugh--Nagumo type systems with the diffusion rate being sufficiently small and/or the reaction term being large enough
(see, e.g., \cite{AHM,ChenXF92,HMN2009,Hilhorst2007,Mottoni,Nakamura,Rubinstein}).
Although there have been many studies regarding interface problems,
it is rather difficult to investigate the global dynamics in presence models.
For this, we shall consider a simplified model but still exhibit abundant patterns.
As a part of our series works, we aim to provide a mathematical framework to study the wave propagation, including colliding of waves, based on a FitzHugh--Nagumo type reaction-diffusion system proposed in \cite{CKN}.

More precisely, we are concerned with a one-dimensional free boundary problem
arising from the following system (\cite{CKN}):
\bea\label{RD}
\begin{cases}
u_t=\Delta u+\dfrac{1}{\ep^2}(f_{\ep}(u)-\ep \beta v),&\quad x\in\mathbb{R}^2,\ t>0,  \\
v_t=g(u,v),&\quad  x\in\mathbb{R}^2,\ t>0,
\end{cases}
\eea
where $\ep>0$ is assumed to be a small parameter; $\beta>0$; $f_{\ep}$ and $g$ take the following form:
\[
 f_{\ep}(u):=u(1-u)\left(u-\dfrac 12+\ep\alpha\right),\quad g(u,v)=g_1u-\dfrac {g_2 v}{g_3v+g_4}
\]
for some $\alpha>0$ and $g_j>0$ for $j=1,2,3,4$.
Let $(u^\ep,v^\ep)$ be a solution of \eqref{RD} with a suitable initial data.
By a formal analysis in \cite{CKN},
$u^{\ep}$ converges to either $1$ or $0$  as $\ep\downarrow  0$.
Denote the region on which $u^{\ep}$ converges to 1 by $\Omega(t)$.
Then, under the assumption
\begin{itemize}
\item[{\bf(A)}]  $g_1g_3>2g_2$,
\end{itemize}
the limiting problem of \eqref{RD} reduces to
\bea\label{SLPn}
\begin{cases}
V=W(v)-\kappa,&\quad (x,y)\in \partial\Omega(t),\ t>0,\\
v_t=g({\bf 1}_{\Omega(t)},v),&\quad (x,y)\in \R^2,\ t>0,
\end{cases}
\eea
where ${\bf 1}_{\Omega(t)}$ stands for the characteristic function having the value $1$ in $\Omega(t)$;
$\kappa$ is the curvature function of $\partial\Omega(t)$; $V$ is the outer normal velocity of $\partial\Omega(t)$, and
\[
W(v)=a-b v,\quad a=\sqrt{2}\alpha,\quad b=6\sqrt 2\beta.
\]
We refer to \cite[Appendix]{CKN} for further details.

The problem \eqref{SLPn} supports many fundamental patterns appearing in excitable media. In \cite{CKN},
traveling spots are considered. The convergence of traveling spots to planar traveling waves
is studied in \cite{CNT}. Traveling curved waves and their stability are investigated in \cite{NW}.
To understand the global dynamics of \eqref{SLPn}, it is natural to { begin with}
the one-dimensional spatial problem so that the curvature vanishes completely.
Therefore, in this paper, we will focus on the following problem:
\bea\label{SLP}
\begin{cases}
V=W(v),&\quad x\in \partial\Omega(t),\ {t>0},\\
v_t=g({\bf 1}_{\Omega(t)},v),&\quad x\in \R,\ {t>0}.
\end{cases}
\eea

In this paper, we always assume {\bf(A)}. In fact, we only require that $g_1g_3>g_2$ in our analysis.
Before studying the global dynamical behavior of \eqref{SLP},
it is necessary to establish the global existence and uniqueness of solutions
of \eqref{SLP} with suitable initial conditions, which will be the primary purpose of this paper.
The global dynamics issue is studied in \cite{CNW2021} as a companion paper.
The weak entire solutions are discussed in \cite{CNW2021-2}.

For initial data of \eqref{SLP}, we assume that
$\Omega(0)=\Omega_0$ contains finitely many disjoint bounded intervals. Namely,
\beaa
\Omega_0:={\bigcup_{j=1}^{m}(x_{2j-1}^0,x_{2j}^0)}\quad \mbox{for some $m\in\mathbb{N}$;}
\eeaa
while the initial function $v_0$ is assumed to be a bounded Lipschitz function defined in $\R$ { and $W(v_0)$ is non-zero on the boundary $\partial\Omega_0$.}
Then we will show that the solution exists in the classical sense until
interfaces collide with each other. The time that two interfaces collide is called an {\em annihilation time}.
This phenomenon has been discussed by Chen and Gao \cite{ChenGao}, who
studied the singular limit problem of the following problem
\bea\label{eq-ac1}
  \left\{
  \begin{array}{rcl}
u_t&=&\ep \Delta u+\dfrac{1}{\ep}(F(u)- v),\quad x\in\mathbb{R},\ t>0, \vspace{0.5mm}\\
v_t&=&G(u,v),\quad x\in\mathbb{R},\ t>0,
  \end{array}
  \right.
\eea
where
\beaa
F(u)=(\frac{3}{\sqrt[3]{2}}-2u^2)u,\quad G(u,v)=u-\gamma v-b\quad \mbox{for some $\gamma>0$ and $b\in\R$}.
\eeaa
For each $v\in(-1,1)$, the equation $F(u)-v=0$ has three real roots: $h_-(v)$,
$h_0(v)$ and $h_+(v)$ satisfying $h_-(v)< h_0(v)<h_+(v)$, where $h_{\pm}(v)$ are the stable
equilibria solution of the ODE $u_t=\ep^{-1}(F(u)- v)$.
The singular limit problem of \eqref{eq-ac1} is described by
\bea\label{SLPn0}
\begin{cases}
V=R(v),&\quad \mbox{in $\partial\Omega_{\pm}(t)$},\\
v_t=G^{\pm}(v):=G(h_{\pm}(v),v),&\quad \mbox{in $\Omega_{\pm}(t)$},
\end{cases}
\eea
where $\Omega_{\pm}(t)$ denotes the region on which $u^{\ep}$ converges to $\pm1$.

The notion of solutions of  \eqref{SLPn0} is extended
as follows: let $D$ be a closed domain in $\R\times [0,\infty)$,
$(v,Q^+,Q^-)$ is called a weak solution of  \eqref{SLPn0}  in $D$
if $(v,Q^+,Q^-)$ satisfies the following conditions:
\begin{enumerate}
\item
$v\in C^0(D)$ and $v_t\in L^{\infty}(D)$,
$v_t=G^{\pm}(v)$ in $Q^{\pm}$,
\item
$Q^{\pm}$ are open and disjoint such that $m(\Gamma)=0$, where $\Gamma:=D\backslash (Q^+\cup Q^-)$  and $m(\cdot)$ denotes the Lebesgue measure in $\R^2$,
\item
(Propagation)
If $B(x_0,r_0)\times \{t_0\}\subset Q^{\pm}$ and %$v\ge 0$
$\pm v<1$ in $B(x_0,r_0+c^{\pm}\delta)\times[t_0,t_0+\delta]\subset D$
for some $\delta>0$, then $B(x_0,r_0+c^{\pm}\delta)\times\{t_0+\delta\}\subset Q^{\pm}$, where
$$c^{\pm}:=\min_{t_0\le t\le t_0+\delta}\{\mp  R(v(x,t))\ |\ x\in \overline{B(x_0,r_0+\delta \sup_{-1<s<1}|R(s)|)} \},\quad
$$
\item
(Nucleation condition) $\{(x,t)\in D|\, \pm v>1\}\subset Q^{\mp}$.
\end{enumerate}

To avoid the confusion of our definition of weak solutions, it may be called a ``switching'' solution.
Based on this setting, they proved that \eqref{SLPn0} admits
a unique ``switching'' solution $(v,Q^+,Q^-)$ that satisfies $v(x,0)=v_0(x)$
with $\Omega_0$ consisting of a finite number of bounded intervals such that
\bea\label{nonzero c}
R(v_0(x))\ne 0\qquad \mbox{ for any }\,  x\in\partial \Omega(0).
\eea
Without the condition given by \eqref{nonzero c}, they also showed the ill-posedness of the problem \eqref{SLPn0}.
Similarly, our problem will be ill-posed with a similar condition (see {\bf (H2)} below).
We remark that \eqref{SLPn0} exhibits the nucleation phenomenon of interfaces and will not appear in our problem.
To study the continuation of solutions after the annihilation time for our problem,
one possible way to discuss weak solutions is to follow the idea of Chen and Gao \cite{ChenGao}. However,
we introduce a different way to define weak solutions,
which is more likely to follow a PDE approach.
We also refer to \cite{ChenXY,ChenXF92,GGI,HNM}
for theoretical works on the existence and uniqueness of solutions
with diffusion term appearing in the $v$-equation.

%%%%%%%%%%%%%%%%%%%%%%%%%%%%%
The rest of the paper is organized as follows. In Section~\ref{sec:results},
we introduce the notion of classical and weak solutions for our model, and state the main results. In Section~\ref{sec:solutions}, we establish the global existence and uniqueness of weak solutions.
Some tedious or straightforward proofs are provided in the Appendix.
%%%%%%%%%%%%%%%%%%%%%%%%%%%%%%%%%%%%%%%%%%%%%%%%%%%%%%%%%%%%%%%%%%%%%%%

\section{Main results}\label{sec:results}
\setcounter{equation}{0}

We consider the following initial value problem:
\bea\label{SLPini}
\begin{cases}
V=W(v):=a-bv,\qquad &x\in \partial\Omega(t),\ {t>0},\vspace{1mm}\\
v_t=g({\bf 1}_{\Omega(t)},v),\qquad &x\in \R,\ {t>0},\vspace{1mm}\\
\Omega(0)=\Omega_0,\quad v(x,0)=v_0(x),\qquad & {x\in\mathbb{R}}.
\end{cases}
\eea
We assume that $(\Omega_0,v_0)$ satisfies
\begin{itemize}
\item[{\bf(H1)}]  (Boundedness) $\Omega_0:={\bigcup_{i=1}^{m}(x_{i}^0,x_{i+1}^0)}$
for some $m\in\mathbb{N}$ and $x_{i}<x_{i+1}$ for $i=1,\cdots,2m-1$,
and $v_0\geq 0$ is a bounded Lipschitz function with
\be\label{ini-bounds}
M:=\|v_0\|_{L^{\infty}(\mathbb{R})}.
\ee
\item[{\bf(H2)}]  (Well-posedness)
$W(v_0(x))\ne 0$ for all $x\in\partial \Omega(0)$.
\end{itemize}

We note that condition {\bf(H2)} is similar to \eqref{nonzero c} used in \cite{ChenGao}. This condition
is used to guarantee the well-posedness of \eqref{SLPini}.
More precisely, the uniqueness of the initial value problem  \eqref{SLPini} may not hold without {\bf(H2)}.
See Remark~\ref{rk-H2} below for the details.

Hereafter, we always define $Q_T:=\mathbb{R}\times[0,T]$.
The definition of classical solutions is given as follows.

\begin{definition}\label{def-classical}
\noindent{\rm(i)} We say that $(\Omega,v)$ is a classical solution of \eqref{SLPini} for $0\leq t\leq T$ if there exist
$x_k\in C^1([0,T])$,  $k=1,...,2m$, and
\beaa
v\in  C(Q_T)\cap C^1\Big(\mathbb{R}\times(0,T]\setminus\{x=x_{k}(t),\ t\in[0,T],\ k=1,...,2m\}\Big)
\eeaa
such that $x_i(\cdot)<x_{i+1}(\cdot)$ in $[0,T]$ for $i=1,\cdots, 2m-1$, and
$\Omega:=\bigcup_{0\leq t\leq T} \left[\Omega(t)\times\{t\}\right]$,
where
\beaa
\Omega(t):=\bigcup_{j=1}^{m}(x_{2j-1}(t),x_{2j}(t)),
\eeaa
and the following equations hold pointwisely:
\bea
&&x'_k(t)=(-1)^{k}W\Big(v(x_{k}(t),t)\Big):=(-1)^{k}\Big(a-bv(x_k(t),t)\Big),\ 0\leq t\leq T,\ k=1,...,2m,\label{ode}\\
&&v_t=g({\bf 1}_{\Omega(t)},v)\quad \mbox{ in }Q_T,\label{pde}\\
&& x_k(0)=x_k^0,\quad v(x,0)=v_0(x).\label{ic}
\eea

\noindent{\rm(ii)}   $(\Omega,v)$ is called a classical solution of \eqref{SLPini} for $0\leq t< T$ if it is a classical solution
for $0\leq t\leq \tau$ for each $\tau\in(0,T)$.

\noindent{\rm(iii)}  $(\Omega,v)$ is called a classical solution of \eqref{SLP} for $\tau\leq t\leq T$ for some $\tau>0$ if {\rm (i)} holds
with $t=0$ replaced by $t=\tau$.

\noindent{\rm(iv)}  $(\Omega,v)$ is called a non-negative classical solution of \eqref{SLP} for $\tau\leq t\leq T$ for some $\tau>0$ if {\rm (iii)} holds and $v\ge 0$.
\end{definition}

Under {\bf(H1)} and {\bf(H2)}, we will establish the {\em local} existence of a classical solution to
problem \eqref{SLPini}, where each interface can be represented by a strictly monotone function of $t$.
The classical solution can be extended until an annihilation occurs and thus the notion of weak solutions is needed. Let us define the annihilation time $T_A$ depending on $(\Omega_0,v_0)$ by
\bea\label{TA}
\quad  T_A:=\sup\{\tau>0|\, x_i(t)<x_{i+1}(t)\ \mbox{$\forall$ $t\in[0,\tau)$ and $i=1,...,2m-1$}\}\in(0,\infty].
\eea
We see that the classical solution exists globally in time if $T_A=\infty$.

Next, we introduce the definition of weak solutions, which is different from the one given in \cite{ChenGao}.
Before we state the definition of weak solutions, we denote
the interior of $\Lambda$ in $\R$ (resp. $\R^2$) by ${\rm int}_\R\Lambda$ (resp. ${\rm int}_{\R^2}\,\Lambda$).
Define the space ${X_T}$ consisting of $(\Omega,v)$ that satisfies the following:
\begin{itemize}
\item[(1)] $v\in C(Q_T)$ and is Lipschitz continuous in $x$,
\item[(2)]  $\Omega\subset Q_T$, $\partial \Omega$ is Lipschitz,
\item[(3)] $v\ne a/b$ on $\overline {\bigcup_{0\le t\le T} \partial\Omega(t)\times\{t\} }$,
\item[(4)] $v(x,0)=v_0(x)$, $\Omega(0)=\Omega_0$,
\end{itemize}
where
\beaa
 \Omega(t):={\rm int}_{\R}\{x\in\R\ |\ (x,t)\in\overline{\Omega}\}.
\eeaa
We remark that $\overline {\bigcup_{0\le t\le T} \partial\Omega(t)\times\{t\} }$ represents the set of all interfaces in $[0,T]$ when it is a classical solution.
Since $\partial \Omega$ is Lipschitz, the unit outer normal vector ${\bf n}:=(n_1,n_2)$ to $\partial \Omega$ exists almost everywhere.

\begin{definition}\label{def-weak}
\noindent{\rm(i)} We say that a pair $(\Omega,v)\in {X_T}$ is a weak solution of \eqref{SLPini} for $0\leq t\leq T$ if
the following two conditions {\bf(C1)} and {\bf(C2)} hold:

{\rm{\bf(C1)}}\quad { For any $\varphi,\psi\in H^1((0,T);L^2(\R))$ and $\psi$ has a compact support in $Q_T$,}
\bea
&&\left[\int_{\R}{\bf 1}_{\Omega(t)}\varphi dx\right]_0^T
=
%\iint_{\Omega}\varphi_tdxdt+
{\int_0^T \int_{\Omega(t)}\varphi_tdxdt+}
\int_{\partial \Omega\cap (\R\times(0,T))}W(v)\varphi |n_1|d\sigma,\label{eq-w1}\\
&&\left[\int_{\R}v\psi dx\right]_0^T
=
\int_0^T\int_{\R}\Big(v\psi_t+g({\bf 1}_{\Omega(t)},v)\psi\Big )dxdt\label{eq-w2}.
\eea
%where
%$n_1$ is the first component of the unit outer normal vector ${\bf n}$.

{\rm{\bf(C2)}}\quad If $B(x_0,r_0)\times\{t_0\}\subset \Omega$ (resp. $\subset\Omega^c:=Q_T\backslash\Omega$) for some $r_0>0$ and $t_0\in[0,T)$, then
there exists $\tau_0\in(0,T-t_0]$ depending only on $r_0$ %and $T-t_0$
such that
\beaa
\{x_0\}\times[t_0,t_0+\tau_0]\subset\Omega \ \mbox{(resp. $\subset\Omega^c$)}.
\eeaa

\noindent{\rm(ii)} We say that a pair $(\Omega,v)$ is a weak solution of \eqref{SLPini} for $0\leq t< T$ if
it is a weak solution for $0\leq t\leq \tau$ for all $\tau\in(0,T)$.
\end{definition}

Note that \eqref{SLPini} could be ill-posed without the condition {\bf(C2)}.
To prevent the nucleation of interfaces, we need an extra condition {\bf(C2)} similar to condition (iii) of { the weak solution in} \cite{ChenGao}.
If $(x_0,t_0)\in \Omega$, then $(x_0,t_0+\ep)\in \Omega$ for sufficiently small $\ep$ owing to the openness of $\Omega$. The above condition means that the slope of the interface $\partial\Omega$ has a positive lower bound in $(x,t)$ space.
Clearly, if any new interface generates from some time $\tau\in(0,T)$, we can choose $\tau_0\ll1$ and $t_0$ close to $\tau$ such that
$t_0+\tau_0>\tau>t_0$ and then
\beaa
\Big(\{x_0\}\times[t_0,t_0+\tau_0]\Big)\cap\Omega\neq\emptyset \mbox{\quad and\quad  } \Big(\{x_0\}\times[t_0,t_0+\tau_0]\Big)\cap\Omega^c\neq\emptyset,
\eeaa
which contradicts {\bf{\bf(C2)}}.

\begin{remark}\label{rem:weak-restricted}
Let $(\Omega,v)$ be a weak solution of \eqref{SLPini} for $0\leq t\leq T$.
Then it is also a weak solution for $0\leq t\leq \tau$ for any $\tau\in (0,T)$ (see Lemma~\ref{lem:wk-t}).
Moreover, for any $0\le t_1<t_2\le T$ and $-\infty\leq x_1<x_2\leq \infty$,
a similar argument used in the proof of Lemma~\ref{lem:wk-t}  given in the Appendix can imply
\beaa
&&\left[\int_{x_1}^{x_2}{\bf 1}_{\Omega(t)}\varphi dx\right]_{t_1}^{t_2}
=
%\iint_{\Omega}\varphi_tdxdt+
{\int_{t_1}^{t_2} \int_{\Omega(t)\cap (x_1,x_2)}\varphi_tdxdt+}
\int_{\partial \Omega\cap ((x_1,x_2)\times(t_1,t_2))}W(v)\varphi |n_1|d\sigma,\label{eq-w1-2}\\
&&\left[\int_{\R}v\psi dx\right]_{t_1}^{t_2}
=
\int_{t_1}^{t_2}\int_{x_1}^{x_2}\Big(v\psi_t+g({\bf 1}_{\Omega(t)},v)\psi\Big )dxdt\label{eq-w2-2}
\eeaa
for any $\varphi,\psi\in H^1((t_1,t_2);L^2(x_1,x_2))$
with $\psi$ has a compact support if $|x_1-x_2|=\infty$.
\end{remark}

%The second result shows the existence and uniqueness of classical solutions.

%\begin{theorem}\label{thm:classical sol}
%Assume {\bf(H1)} and {\bf(H2)}. Then, problem \eqref{SLPini} has a unique classical solution for $0\leq t< T_A$,
%where $T_A$ is defined in \eqref{TA}.
%Moreover, the following hold:
%\begin{itemize}
%\item[{\rm (i)}] $W(v(x_k(t),t))$ never changes its sign for $t\in[0,T_A)$. More precisely, there exists a positive constant $\delta$ such that
%\beaa
%|W(v(x_k(t),t))|\geq \delta \quad \mbox{{\rm(}or, equivalently,  $|x_k'(t)|\geq \delta${\rm)}}
%\eeaa
%{ for all $t\in[0,T_A)$ and $k=1,...,2m$.}
%\item[{\rm (ii)}] The classical solution $(\Omega,v)$ can extend as a weak solution for $0\leq t\leq T_A$.
%\item[{\rm (iii)}] Problem \eqref{SLPini} has a unique classical solution with initial time $t=T_A$ and the solution can be extended until the next %annihilation occurs.
%\end{itemize}
%\end{theorem}

%\begin{remark}\label{rk-collide}
%We see from Theorem~\ref{thm:classical sol}(i)
%that the annihilation must occur if  $x_{2j}'(0)<0$
%and $x_{2j-1}'(0)>0$, or $x_{2j}'(0)>0$
%and $x_{2j+1}'(0)<0$ for some $j\in\{1,...,m-1\}$.
%\end{remark}

We now state the main results as follows.

\begin{theorem}\label{thm:local existence}
Assume {\bf(H1)} and {\bf(H2)}. Then
problem \eqref{SLPini} has a unique local in time
non-negative classical solution.
\end{theorem}

\begin{theorem}\label{thm:global weak sol}
Assume {\bf(H1)} and {\bf(H2)}. Then there is a unique global in time
weak solution to problem \eqref{SLPini}.
\end{theorem}

%%%%%%%%%%%%%%%%%%%%%%%%%%%%%%%%%%%%%%%%%%%%%%%%%%%%%%%%%%%%%%%%%%%%%%

%%%%%%%%%%%%%%%%%%%%%%%%%%%%%%%%%%%%%%%%%%%%%%%%%%%%%%%%%%%%%%%%%%%%%%%
%\section{Preliminaries}
\section{Classical solutions and weak solutions}\label{sec:solutions}
\setcounter{equation}{0}

We shall divide this section into two subsections.
In the first subsection, we study the local existence and uniqueness of classical solutions and
prove Theorem~\ref{thm:local existence}.
In the second subsection, we establish the global existence and uniqueness of weak solutions
(Theorem~\ref{thm:global weak sol}).
Some tedious proofs are put in the Appendix. Hereafter, {\bf(H1)} and {\bf(H2)} are always assumed.

%%%%%%%%%%%%%%%%%%%%%%%%%%%%%%%%%%%%%%%%%%%%%%%%%%
\subsection{The local existence and uniqueness of classical solutions}
%The proof of Theorem~\ref{thm:local existence} is rather tedious and
%so we put some proofs in the Appendix.

First, we deal with the local existence of solutions.

\begin{proposition}\label{prop:local-existence}
Problem \eqref{SLPini} has a local in time classical solution.
\end{proposition}

The key point of the proof of Proposition~\ref{prop:local-existence}  is
the notion of the arrival time (cf.\cite{NW}), which is given as follows.

\begin{definition}\label{def:arrival time}
For a classical solution $(\Omega,v)$ and a strictly monotone interface $x=x_k(\cdot)$ for some $k\in\{1,...,2m\}$,
we say that $T_k(y)$ is the arrival time of the interface $x=x_k(t)$ to some given $y\in\R$
if $y=x_k(T_k(y))$. For convenience, we define $T_k(y):=0$ if $y\leq x_k(0)$ with $x_k'(0)>0$ or $y\geq x_k(0)$ with $x_k'(0)<0$.
\end{definition}

The arrival time $t=T_k(y)$ can be almost viewed as the inverse function of $t=x_k^{-1}(y)$. However,
in our definition, the arrival time can be always defined as $0$ in some $y$.
By the help of the arrival time, we can calculate $v(x,t)$ in terms of functions defined in \cite{CKN,CNT,NW}:
\bea\label{G func}
G_0^{-1}(v):=\displaystyle\int_M^{v}\dfrac{d\xi}{g(0,\xi)},\quad
G_1^{-1}(v):=\displaystyle\int_0^{v}\dfrac{d\xi}{g(1,\xi)},
\eea
where $M$ is given in \eqref{ini-bounds} below.
{ For the basic properties of $G_0$ and $G_1$, we state the following two lemmas.
%%%%%%%%%%%%%%%%%%%%%%%%%%%%%%%%%%%%%%%%%%%%%%%%%
\begin{lemma}[{\cite[Lemma 2.1]{NW}}]\label{lem-G}
The functions $G_0$ and $G_1$ defined in \eqref{G func} satisfy
\begin{itemize}
\item[{\rm (i)}] $\ 0\leq G_0(G_0^{-1}(s)+t)\leq s,\quad s,t\ge 0,$
\item[{\rm (ii)}] $\ \dfrac {d}{ds}G_0(G_0^{-1}(s)+t)=
\begin{cases}
\dfrac {g(0,G_0(G_0^{-1}(s)+t))}{g(0,s)},&s>0,\ t\ge 0,\\
e^{-g_2t/g_4},& s=0,\ t\ge 0,
\end{cases}$
\item[{\rm (iii)}] $\ \dfrac {d}{dt}G_1(G_1^{-1}(s)+t)=g(1,G_1(G_1^{-1}(s)+t)),\quad s,t\ge 0.$
%\eeaa
\end{itemize}
\end{lemma}
%%%%%%%%%%%%%%%%%%%%%%%%%%%%%%%%%%%%%%%%%%%%%%%%%%

%%%%%%%%%%%%%%%%%%%%%%%%%%%%%%%%%%%%%%%%%%%%%%%%%%
\begin{lemma}\label{lem:G0 Lip}
There holds that
\beaa
&&\Big|G_0\Big(G_0^{-1}(u)+t\Big)-G_0\Big(G_0^{-1}(v)+t\Big)\Big|\leq |u-v|,\quad u,v\geq0,\ t\geq0.
\eeaa
\end{lemma}
\begin{proof}
Lemma~\ref{lem-G} (i) and (ii) imply that
\beaa
\Big|\dfrac {d}{ds}G_0(G_0^{-1}(s)+t)\Big|\leq1,\quad s,t\ge 0.
\eeaa
This lemma immediately follows from the above fact and the mean value theorem.
%It follows from the mean value theorem and { the fact that
%\beaa
%\Big|\dfrac {d}{ds}G_0(G_0^{-1}(s)+t)\Big|\leq1,\quad s,t\ge 0,
%\eeaa
%which follows from}
%Lemma~\ref{lem-G} (i)(ii).
\end{proof}
}
%%%%%%%%%%%%%%%%%%%%%%%%%%%%%%%%%%%%%%%%%%%%%%%%%%
\begin{proof}[Proof of Proposition~\ref{prop:local-existence}]
To simplify the proof, we only consider $m=1$, i.e., $\Omega_0=(x_1^0,x_2^0)$. The following process can apply to $m>1$ with some simple modifications but tedious details.
Because of {\bf(H2)}, we can divide our discussion into four cases:
\begin{itemize}
\item[(1)]  $W(v_0(x_{1}^0))>0$ and $W(v_0(x_{2}^0))>0$,
\item[(2)]  $W(v_0(x_{1}^0))>0$ and $W(v_0(x_{2}^0))<0$,
\item[(3)]  $W(v_0(x_{1}^0))<0$ and $W(v_0(x_{2}^0))>0$,
\item[(4)]  $W(v_0(x_{1}^0))<0$ and $W(v_0(x_{2}^0))<0$.
\end{itemize}

For the case (1), first we assume in advance that $x_{k}(t)$ exists $(k=1,2)$.
By the continuity of $v$ and $W$, we see that $x_1'(t)<0$ and $x_2'(t)>0$ for $t\in[0,\tau)$ for some $\tau>0$ sufficiently small.
By \eqref{SLPini}, we have
$v_t=g(0,v)$ for $0<t<\tau$ and $x\in(-\infty,x_1(t)]\cup [x_2(t),\infty)$.
By dividing both  sides by $g(0,v)$ and integrating it over $[0,t]$, we can easily solve $v$ as
\[
v(x,t)=G_0(G_0^{-1}(v_0(x))+t)\quad \mbox{for $0<t<\tau$ and $x\in(-\infty,x_1(t)]\cup [x_2(t),\infty)$},
\]
where $G_0^{-1}$ is defined in \eqref{G func}. Thus \eqref{ode} reduces to
\[
\dfrac {dx_{k}}{dt}=(-1)^k W(G_0(G_0^{-1}(v_0(x_{k}(t)))+t)),\quad x_k(0)=x_k^0.
\]
%{ It follows from the Lipschitz continuity of $v_0$ and Lemma \ref{lem:G0 Lip} that the right hand side is Lipschitz continuous with respect to $x_k$.}
With the help of Lemma~\ref{lem:G0 Lip} and the Lipschitz continuity of $v_0$, the above initial value problem allows
 us to define the position of $x_{k}$ $(k=1,2)$ uniquely for all small $t\in[0,\tau')$ for some $\tau'<\tau$.
 %due to the continuity of $v$.
Finally,  for $(x,t)\in(x_1(t),x_2(t))\times(0,\tau')$, $v$ can be solved by integrating $v_t/g(1,v)=1$. Namely,
\beaa
v(x,t)=\begin{cases}
   G_1\left(G_1^{-1}(v_0(x))+t\right),&\ x_1^0\leq x \leq x_2^0,\ t\in(0,\tau'),\\
   G_1\left(G_1^{-1}(v(x,T_1(x)))+t-T_1(x)\right),&\ x_1(t)<x<x_1^0,\ t\in(0,\tau'),\\
   G_1\left(G_1^{-1}(v(x,T_2(x)))+t-T_2(x)\right),&\ x_2^0<x<x_2(t),\ t\in(0,\tau'),
   \end{cases}
\eeaa
where $T_k(x)$ is the arrival time of $x_k(t)$ to $x$.

Hence we obtain the local existence and %the uniqueness
of a classical solution of \eqref{SLPini} for the case (1).
The similar process can apply to cases (2), (3) and (4) respectively as well as the case where $m\ge 2$. We omit the details.
This completes the proof.
\end{proof}
%%%%%%%%%%%%%%%%%%%%%%%%%%%%%%%%%%%%%%%%%%%%%%%%%%

%Lemma~\ref{thm:local existence} does not imply the uniqueness of a classical solution yet, since it only gives the existence of a local in time solution.
Next, we deal with the uniqueness and continuation
of solutions. To extend the local in time solution uniquely, we need the Lipschitz continuity of $v$. For this, we prepare several lemmas.
%In order to show the unique existence of a classical solution and the Lipschitz continuity of $v$, we prepare several lemmas.

%%%%%%%%%%%%%%%%%%%%%%%%%%%%%%%%%%%%%%%%%%%%%%%%%%
\begin{lemma}\label{lem:lipschitz}
%Assume {\bf(H1)} and {\bf(H2)}.
Let $(\Omega,v)$ be a classical solution of \eqref{SLPini} for $0\leq t\leq T$.
Furthermore, assume that there exists $\delta>0$ such that
\bea\label{phi-monotone}
|{x}_{k}'(t)|\geq \delta\quad \mbox{for  $0\leq t\leq T$ and $k=1,...,2m$}.
\eea
Then $v$ is Lipschitz continuous on $Q_T:=\mathbb{R}\times[0,T]$,
where the Lipschitz constant depends on $\delta$.
\end{lemma}
%%%%%%%%%%%%%%%%%%%%%%%%%%%%%%%%%%%%%%%%%%%%%%%%%%
%%%%%%%%%%%%%%%%%%%%%%%%%%%%%%%%%%%%%%%%%%%%%%%%%%
\begin{proof}
The proof is involved because
each point $x$ may be passed through by several interfaces during a period of time.
We simply separate $Q_T$ into finitely many adjacent closed regions such that at most one interface can pass through any points and then show Lipschitz continuity on each closed region.

Recall the definition of classical solutions, we have ${x}_{k}(t)<{x}_{k+1}(t)$ for $t\in[0,T]$ and ${x}_k\in C^1([0,T])$ for $k=1,...,2m-1$.
Hence we can take $A=\min_{t\in[0,T]}{x}_1(t)$ and $B=\max_{t\in[0,T]}{x}_{2m}(t)$ such that
$\Omega\subset [A,B]\times[0,T]$.
In other words, we have
\beaa
v_t=g(0,v) \quad \mbox{in $D_T:=Q_T\setminus ([A,B]\times[0,T])$},
\eeaa
which gives (see the proof of Proposition~\ref{prop:local-existence})
\bea\label{v form-G0}
v(x,t)=G_0\Big(G_0^{-1}(v_0(x))+t\Big),\quad (x,t)\in D_T.
\eea

We now show that
\bea\label{Lip-outside}
\mbox{$v$ is Lipschitz continuous on $D_T$.}
\eea
Clearly, $v$ is Lipschitz continuous in $t$. From \eqref{v form-G0}  and {\bf (H1)}, we can use Lemma~\ref{lem:G0 Lip} and the Lipschitz continuity of $v_0$ to assert
\beaa
|v(x,t)-v(\bar{x},t)|\leq |v_0(x)-v_0(\bar{x})|\leq L_0|x-\bar{x}|
\eeaa
for all $(x,t),(\bar{x},t)\in D_T$ and for some $L_0>0$. Hence \eqref{Lip-outside} follows.

Next, we show that
\bea\label{Lip-inside}
\mbox{$v$ is Lipschitz continuous on $[A,B]\times[0,T]$.}
\eea
To simplify our discussion, we write $[0,T]=[0,\tau]\cup[\tau,2\tau]\cup\cdots\cup[(N-1)\tau, T]$, where $\tau:=T/N$ for some $N\in\mathbb{N}$ large enough such that
\bea\label{non-overlapping}
{x}_j([(n-1)\tau,n\tau])\cap {x}_k([(n-1)\tau,n\tau])=\emptyset\quad \mbox{for all $n=1,...,N$ and $j\neq k$.}
\eea
To prove \eqref{Lip-inside}, it suffices to show
\bea\label{Lip-inside-tau}
\mbox{$v$ is Lipschitz continuous on $[A,B]\times[(n-1)\tau,n\tau]$ for $n=1,...,N$.}
\eea
By Lemma~\ref{lem:lipschitz-appendix}, we see that \eqref{Lip-inside-tau} follows for $n=1$ with the Lipschitz constant depending on $\delta$.
Repeating the same argument used in the proof of Lemma~\ref{lem:lipschitz-appendix}, we obtain \eqref{Lip-inside-tau} and then
\eqref{Lip-inside} holds. Together with \eqref{Lip-outside}, we thus complete the proof of Lemma~\ref{lem:lipschitz}.
\end{proof}
%%%%%%%%%%%%%%%%%%%%%%%%%%%%%%%%%%%%%%%%%%%%%%%%%%

{
%%%%%%%%%%%%%%%%%%%%%%%%%%%%%%%%%%%%%%%%%%%%%%%%%%
\begin{lemma}\label{lem:lipschitz-appendix}
Under the same hypothesis to Lemma~\ref{lem:lipschitz},
$v$ is Lipschitz continuous on $[A,B]\times[0,\tau]$,
where $A$, $B$ and $\tau$ are defined in the proof of Lemma~\ref{lem:lipschitz}.
\end{lemma}
%%%%%%%%%%%%%%%%%%%%%%%%%%%%%%%%%%%%%%%%%%%%%%%%%%
%%%%%%%%%%%%%%%%%%%%%%%%%%%%%%%%%%%%%%%%%%%%%%%%%%
\begin{lemma}\label{lem:A}
%Assume {\bf(H1)} and {\bf(H2)}. Then
The local in time classical solution of problem \eqref{SLPini} is unique. Moreover,
the solution can be extended until $x_k'$ vanishes at some time for some $k$ or an annihilation occurs.
\end{lemma}
%%%%%%%%%%%%%%%%%%%%%%%%%%%%%%%%%%%%%%%%%%%%%%%%%%

The proofs of the above lemmas are put in the Appendix.}

%%%%%%%%%%%%%%%%%%%%%%%%%%%%%%%%%%%%%%%%%%%%%%%%%%
%%%%%%%%%%%%%%%%%%%%%%%%%%%%%%%%%%%%%%%%%%%%%%%%%%
%%%%%%%%%%%%%%%%%%%%%%%%%%%%%%%%%%%%%%%%%%%%%%%%%%
\begin{proposition}\label{prop:classical sol}
%Assume {\bf(H1)} and {\bf(H2)}. Then
The classical solution of problem \eqref{SLPini} can be extended uniquely until an annihilation occurs.
Moreover,
$x_k(t)$ is strictly monotone in $[0,T_A)$, where $T_A$ is defined in \eqref{TA}. If $T_A<\infty$,
there exists a positive constant $\delta$ such that $|x_k'(t)|\geq \delta$ for all $t\in[0,T_A)$ and $k=1,...,2m$.
\end{proposition}
%%%%%%%%%%%%%%%%%%%%%%%%%%%%%%%%%%%%%%%%%%%%%%%%%%
%%%%%%%%%%%%%%%%%%%%%%%%%%%%%%%%%%%%%%%%%%%%%%%%%%
\begin{proof}
By Lemma \ref{lem:A}, the classical solution can be extended uniquely
until $x_k'(\tau_1)=0$\linebreak
(or $W(v(x_k(\tau_1),\tau_1))=0$) for some $\tau_1>\tau_0$ (non-uniqueness will occur) or
$x_k$ intersects $x_{k+1}$ for some $k$ at some time (an annihilation occurs).

We now prove the existence of $\delta$  when $T_A<\infty$ by using a contradiction argument. Assume that there exist an increasing sequence $\{t_j\}$ and
some $k$ such that $W(v(x_k(t_j)),t_j)\to 0$ as $j\to\infty$. Then we can divide our discussion into two cases:
\beaa
\mbox{{\bf Case 1}:\quad $t_j\uparrow \tau $ for some $\tau\in(0,T_A)$ as $j\to\infty$},
\quad\mbox{{\bf Case 2}:\quad  $t_j\uparrow T_A$ as $j\to\infty$.}
\eeaa

We now consider {\bf Case 1}. In this case, we can assume that there exists $k\in\{1,...,2m\}$ such that
\beaa
&&W(v(x_k(\tau),\tau))=0,\quad W(v(x_k(t),t))\neq0 \quad \mbox{for $t\in[0,\tau)$}.
\eeaa
We shall divide our discussion into four subcases:
\begin{itemize}
\item[(1-a)] $k$ is odd and $W(v_0(x_k^0))<0$,
\item[(1-b)] $k$ is odd and $W(v_0(x_k^0))>0$,
\item[(1-c)] $k$ is even and $W(v_0(x_k^0))<0$,
\item[(1-d)] $k$ is even and $W(v_0(x_k^0))>0$.
\end{itemize}

First we consider the subcase (1-a). By the definition of $W$ and $\Omega$, we have
\bea\label{contra}
v(x_k(t),t)>\frac{a}{b},\quad \mbox{and $x_k'(t)>0$ for $t\in[0,\tau)$}.
\eea
Hence the arrival time $t=T_k(x)$ is well-defined for $x\in(x_k^0,y)$ with $y:=x_k(\tau)$ and $\tau=T_k(y)$.
Now we take $\varepsilon>0$ sufficiently small and define
\beaa
\tau_\ep:=T_k(y-\varepsilon)<T_k(y)=\tau,\quad
D_{\varepsilon}:=\{(x,t)|\, x_k(t)< x \leq y,\ \tau_\ep\leq  t< \tau\}.
\eeaa
Note that we can choose $\varepsilon>0$ sufficiently small such that
$D_{\varepsilon}\subset \Omega$ (excitation region),
which means
\bea\label{monotone}
\mbox{$v(x,t)$ is strictly increasing in $t$ for all $(x,t)\in D_{\varepsilon}$}.
\eea
In particular,
\beaa
v(y,t)\uparrow \frac{a}{b}\quad  \mbox{as $t\uparrow \tau$ for $\tau_\ep\leq  t< \tau$}.
\eeaa
It follows that $v(y,\tau_\ep)<a/b$. On the other hand, by \eqref{contra}, we have $v(y-\ep,\tau_\ep)>a/b$.
By the continuity of $v$, there exists $x_0\in (y-\ep,y)$ such that
$v(x_0,\tau_\ep)$ must attain at $a/b$. Hence
we can define the following point in $(y-\ep,y)$:
\[
x_\ep:=\sup\Big\{x\in (y-\ep,y)\ \Big|\ v(x,\tau_\ep)=\dfrac ab\Big\}.
\]
Because of \eqref{monotone},
we can introduce the notion of the attaining time $\beta(x)\leq T_k(x)$ satisfying
\[
v(x,\beta(x))=\dfrac ab \quad \mbox{for each $x\in[x_\ep,y)$.}
\]
See Figure~\ref{fig:2}.
Then we can show that $\beta$ is Lipschitz on $[x_\varepsilon,y]$.
To do so, using $v_t/g(1,v)=1$ in $D_{\varepsilon}$ and integrating it over $[\tau_\ep, \beta(x)]$ give
\beaa
\int_{v(x,\tau_\ep)}^{a/b}\frac{ds}{g(1,s)}=\beta(x)-\tau_\ep.
\eeaa
Since $v(\cdot,\tau_\ep)$ is Lipschitz on $[x_\varepsilon,y]$ (because of Lemma~\ref{lem:lipschitz}),
we see that $\beta(\cdot)$ is Lipschitz on $[x_\varepsilon,y]$ with the Lipschitz constant, say $L$.
Thus using $T_k(y)=\beta(y)$ and $T_k(x)>\beta(x)$ for $x\in[x_\ep,y)$, we have
\bea\label{L-ep'}
0\le T_k(y)-T_k(y-\ep')\le\beta(y)-\beta(y-\ep')\le L\ep'
\eea
for any small $\ep'\in(0,\ep)$.
On the other hand, since $x_k'(\tau)=0$ and $x_k'(t)>0$ for $0\leq t<\tau$, we have $T'_k(y^-)=+\infty$.
This reaches a contradiction with \eqref{L-ep'}.
Hence we have shown the existence of $\delta$ for the subcase (1-a).
The argument used in the proof of { (1-a)}
can apply to subcases (1-b), (1-c), (1-d). We omit the details.

Next, we deal with {\bf Case 2}. Because of {\bf Case 1}, there exists $k\in\{1,...,2m\}$ such that
\beaa
&&W(v(x_k(T_A^-),T_A^-))=0,\quad W(v(x_k(t),t))\neq0 \quad \mbox{for $t\in[0,T_A)$}.
\eeaa
Since each interface $x=x_k(t)$ is monotone in $t$ and bounded for $[0,T_A)$ because of $T_A<\infty$,
$x_k(T_A^-)$ exists and is finite, then $\lim_{t\to T_A^-} v(x,t)$ exists and is finite for each $x\in\mathbb{R}$, which means that
$v$ can be extended continuously to $t=T_A$ as in the proof of Proposition~\ref{prop:local-existence}.
Hence we can define the arrival time $T_k(\cdot)$ on $(x_k^0,x_k(T_A^-)]$ for each $k$.
This allows us to use the same argument as in {\bf Case 1} with $\tau$ replaced by $T_A$ to complete the proof of {\bf Case 2}.

Finally, note that if $T_A=\infty$, from the argument of {\bf Case 1}, we see that $x'_k$ never vanishes in $[0,\infty)$
and then $x_k$ is strictly monotone.
Hence the proof of Proposition~\ref{prop:classical sol} is complete.
\end{proof}

%%%%%%%%%%%%%%%%%%%%%%%%%%%%%%%%%%%%%%%%%%%%%%%%%%

\begin{figure}
% Use the relevant command to insert your figure file.
% For example, with the graphicx package use
  \centering\includegraphics[width=0.4\textwidth]{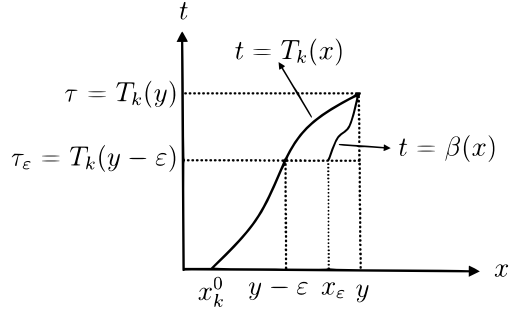}
% figure caption is below the figure
\caption{A diagram of $t=\beta(x)$.}
\label{fig:2}       % Give a unique label
\end{figure}
%%%%%%%%%%%%%%%%%%%%%%%%%%%%%%%%%%%%%%%%%%%%%%%%%%

%%%%%%%%%%%%%%%%%%%%%%%%%%%%%%%%%%%%%%%%%%%%%%%%%%
%\begin{remark}\label{rem:vatTA}
%By Proposition~\ref{prop:classical sol} (i),
%we have
% $|x_{k}'(T_A^-)|\geq \delta$
%for some $\delta>0$.
%This immediately implies that
%\[
%\lim_{t\uparrow T_A}v(x,t)\ne \dfrac ab\quad \mbox{ if $T_A<\infty$}.
%\]
%%Namely, {\bf (H2)} holds true at $t=T_A^-$.
%\end{remark}
%%%%%%%%%%%%%%%%%%%%%%%%%%%%%%%%%%%%%%%%%%%%%%%%%%

{
Next we show that a classical solution becomes a weak one.
Since $x_i(t)$ is monotone in time for each $i$, { we} can define
\beaa
\lim_{t\to T^-}{\Omega}(t):={\bigcup_{j=1}^{m}\Big(\lim_{t\to T^-}x_{2j-1}(t),\lim_{t\to T^-}x_{2j}(t)\Big)},
\eeaa
when ${\Omega}(t)={\bigcup_{j=1}^{m}(x_{2j-1}(t),x_{2j}(t))}$.
% Note that ${\rm int}_\R\,[a,b]=(a,b)$ and  ${\rm int}_\R\,([a,b]\cup[b,c]=(a,c)$ if $a<b<c$.
%{\bf ?? I added the definition of ${\rm int}_\R$. But if you two think that it is not necessary, please remove it. We didn't define the closure and the complement set neither ??}

%%%%%%%%%%%%%%%%%%%%%%%%%%%%%%%%%%%%%%%%%%%%%%%%%%
\begin{proposition}\label{prop:cw}
Let $({\Omega},{v})$ be a classical solution of \eqref{SLPini} for $0\leq t< T_A$ with $T_A<\infty$. Then
$(\widetilde{\Omega},\widetilde{v})$ is a weak solution for $0\leq t\leq T_A$, where
\bea
&&\widetilde{\Omega}(t):= \begin{cases}
                \Omega(t)\qquad &\mbox{for $t\in[0,T_A)$},\\
                 {\rm int}_{\mathbb{R}}\Big(\lim_{t\to T_A^-} \overline{\Omega(t)}\Big)\qquad &\mbox{for $t=T_A$},
                 \end{cases}\label{new omega}\\
&&\widetilde{v}(\cdot,t):=\begin{cases}
                v(\cdot,t)\qquad &\mbox{for $t\in[0,T_A)$},\\
                \lim_{t\to T_A^-} v(\cdot,t)\qquad &\mbox{for $t=T_A$}.
                 \end{cases}\label{new v}
\eea
In particular, {\bf(H1)} and {\bf(H2)} holds with $(\Omega_0,v_0)$ replaced by $(\widetilde{\Omega}(T_A),\widetilde{v}(x,T_A))$.
Hence there exists a unique classical solution with initial time $t=T_A$ and the solution can be extended until the next annihilation occurs.
\end{proposition}
%%%%%%%%%%%%%%%%%%%%%%%%%%%%%%%%%%%%%%%%%%%%%%%%%%
%%%%%%%%%%%%%%%%%%%%%%%%%%%%%%%%%%%%%%%%%%%%%%%%%%
\begin{proof}
Recall that ${\Omega}(t):={\bigcup_{j=1}^{m}(x_{2j-1}(t),x_{2j}(t))}$. Since $x_k(\cdot)$ is monotone and bounded for each $k$, we see that $x_k(T_A^-)$ exists and thus $\widetilde{\Omega}(T_A)$ is well-defined.
Since $g({\bf 1}_{\Omega(t)},v)\le g_1$,
$\widetilde{v}(\cdot,T_A)$ is bounded in $\mathbb{R}$.
 Note that, for each fixed $x$, there exists small $\epsilon>0$ such that
$\widetilde{v}(x,t)$ is monotone in $t$ for $t\in(T_A-\epsilon, T_A)$. Thus,
$\widetilde{v}(\cdot,T_A):=\lim_{t\to T_A^-} v(\cdot,t)$ is well-defined. Moreover, $\widetilde{v}$ is Lipschitz in $x$ by
using the proof of Proposition~\ref{prop:classical sol} and Lemma~\ref{lem:lipschitz}.
Also, from the definition of \eqref{new omega},
we see that $\widetilde{\Omega}(t)$ consists of finitely many disjoint bounded intervals for $t\in[0,T_A]$.
It follows that $(\widetilde{\Omega},\widetilde{v})\in X_{T_A}$.
Clearly,
\eqref{eq-w1} and \eqref{eq-w2} hold with $T$ replaced by $T_A$.
Also, since $(\widetilde{\Omega},\widetilde{v})$ is a classical solution for $0\leq t< T_A<\infty$,
the condition {\bf(C2)} is satisfied with $T$ replaced by $T_A$.
Hence,
$(\widetilde{\Omega},\widetilde{v})$ is a weak solution for $0\leq t\leq T_A$.
In particular,
{\bf(H1)} and {\bf(H2)} holds with
$(\Omega_0,v_0)$ replaced by $(\widetilde{\Omega}(T_A),\widetilde{v}(x,T_A))$.
By taking $t=T_A$ as an initial time and by applying Proposition~\ref{prop:classical sol},
% to the problem \eqref{SLPini} starting from $t=T_A$,
%\eqref{SLP},
we can assert that there exists a unique classical solution with initial time $t=T_A$.
Namely, the solution can be extended until the next annihilation occurs.
Therefore, the proof is completed.
\end{proof}

%%%%%%%%%%%%%%%%%%%%%%%%%%%%%%%%%%%%%%%%%%%%%%%%%%
\begin{lemma}\label{lem-3.1}
Let $(\Omega,v)$ be a weak solution of \eqref{SLPini} for $0\leq t\leq T$.
If $v_0(x)\geq0$ for all $x\in\R$, then $v(x,t)\geq0$  in $\R\times[0,T]$.
\end{lemma}
%%%%%%%%%%%%%%%%%%%%%%%%%%%%%%%%%%%%%%%%%%%%%%%%%%

The proof of this lemma is stated in the Appendx.}
We are ready to prove Theorem~\ref{thm:local existence}.

\begin{proof}[Proof of Theorem~\ref{thm:local existence}]
By Proposition~\ref{prop:local-existence} and Proposition~\ref{prop:classical sol},
we obtain the local existence and uniqueness
of a classical solution. Moreover, by Lemma~\ref{lem-3.1},
the classical solution (also a weak solution) is non-negative. This complete the proof.
\end{proof}

\begin{remark}\label{rk-H2}
The condition {\bf(H2)} is necessary for the uniqueness.
If there is a point $x_0\in\partial \Omega(0)$ such that $W(v_0(x_0))= 0$, then the uniqueness does not hold.
Indeed, we can construct the multiple solutions
starting from the initial condition $(\Omega(0), v_0(x))$ as follows:
\beaa
\Omega(0)=\{x\in\R\ |\ x>0\},\quad v_0(x)=\dfrac ab-\arctan x.
\eeaa
Then we have two solutions:
\belaa
\Omega_1(t)=\{x\in\R\ |\ x>s_1(t)\},\\
v_1(x,t)=\begin{cases}
   G_1\left(G_1^{-1}(v_0(x))+t\right),\qquad (x>s_1(t)),\\
   G_0\left(G_0^{-1}(v_0(x))+t\right),\qquad (x\le s_1(t)),
   \end{cases}\\
s_1'(t)=a-bG_1\left(G_1^{-1}(v_0(s_1(t)))+t\right),
\eelaa
and
\belaa
\Omega_2(t)=\{x\in\R\ |\ x>s_2(t)\},\\
v_2(x,t)=\begin{cases}
   G_1\left(G_1^{-1}(v_0(x))+t\right),\qquad (x>s_2(t)),\\
   G_0\left(G_0^{-1}(v_0(x))+t\right),\qquad (x\le s_2(t)),
   \end{cases}\\
s_2'(t)=a-bG_0\left(G_0^{-1}(v_0(s_2(t)))+t\right).
\eelaa
We regard the interface as a front for the first solution and a back for the second solution.
Note that $g(1,a/b)>0$ by the assumption {\bf(A)}. Thus $s_1(t)$ is decreasing in $t$ and $s_2(t)$ is increasing in $t$ for $t$ close to zero.
We can also construct other solutions.
Thus the condition {\bf(H2)} is required for the uniqueness of solutions
as well as the well-posedness of solutions. We refer to \cite{ChenGao} for more detailed discussion.
\end{remark}

\subsection{The global existence and uniqueness of weak solutions}

In this subsection, we shall establish the global existence and uniqueness of weak solutions to problem \eqref{SLPini}.
Since $T_A=\infty$ means the classical solution exists globally in time (so does the weak solution),
we only discuss weak solutions when $T_A<\infty$.
%Unfortunately, conditions {\bf(H1)} and {\bf(H2)} cannot guarantee
%the uniqueness of weak solutions since new interfaces may generate at some positive time.
%Hence we need an extra condition {\bf(C2)} in Definition~\ref{def-weak}
%to prevent that occurs.

{
%%%%%%%%%%%%%%%%%%%%%%%%%%%%%%%%%%%%%%%%%%%%%%%%%%
\begin{lemma}\label{lem:wk-t}
Let $(\Omega,v)$ be a weak solution of \eqref{SLPini} for $0\leq t\leq T$.
Then it is also a weak solution of \eqref{SLPini} for $0\leq t\leq \tau$ for any $\tau\in (0,T)$.
\end{lemma}
%%%%%%%%%%%%%%%%%%%%%%%%%%%%%%%%%%%%%%%%%%%%%%%%%%

The converse of Lemma \ref{lem:wk-t} does not hold in general because $v$ may attain $a/b$ at $t=T$ or $\partial\Omega$ is not Lipschitz continuous at $t=T$ (such that $(\Omega,v)\not\in {X_T}$). However, in our problem, the weak solution exists globally in time (see Proposition \ref{prop:global weak sol}).
}
%%%%%%%%%%%%%%%%%%%%%%%%%%%%%%%%%%%%%%%%%%%%%%%%%%
\begin{lemma}\label{lem:extend}
Let $(\Omega_1,v_1)$ {\rm (}resp. $(\Omega_2,v_2)${\rm)} { be} a weak solution for $0\leq t\leq T_1$ {\rm (}resp. $T_1\leq t\leq T_2${\rm )}.
Define
\bea
\widehat{\Omega}&:=&\ds{\rm int}_{\R^2}\,\overline{\Omega_1\cup\Omega_2},\label{ext-omega}\\
\widehat{v}&:=&\begin{cases}v_1\quad &(0\leq t\leq T_1),\\
v_2&(T_1\leq t\leq T_2).\end{cases} \label{ext-v}
\eea
If
\[
{\rm int}_\R\,\lim_{t\to T_1^-}\overline {\Omega_1(t)}=\Omega_2(T_1),\quad v_1(x,T_1)=v_2(x,T_1)\quad \mbox{for any $x\in\R$},
\]
then $(\widehat{\Omega},\widehat{v})$ is a weak solution for $0\leq t\leq T_2$.
\end{lemma}
%%%%%%%%%%%%%%%%%%%%%%%%%%%%%%%%%%%%%%%%%%%%%%%%%%

{ The proofs of two lemmas are} put in the Appendix. We now provide a simple example to illustrate Lemma~\ref{lem:extend}.
Let $(\Omega_1,v_1)$ {\rm (}resp. $(\Omega_2,v_2)${\rm)} a classical solution for $0\leq t<T_1$ {\rm (}resp. $T_1\leq t<T_2${\rm )}
such that
\beaa
\Omega_1(t)&:=&(x_1(t),x_2(t))\cup (x_3(t),x_4(t))\qquad (0\leq t<T_1),\\
\Omega_2(t)&:=&(x_1(t),x_4(t))\hspace{3.5cm} (T_1\leq t<T_2).
\eeaa
Also, we assume that
\beaa
x_1(T_1)<x_2(T_1)=x_3(T_1)<x_4(T_1)\quad \mbox{and $v_1(x,T_1)=v_2(x,T_1)$ for any $x\in\R$}.
\eeaa
By Proposition~\ref{prop:cw}, two classical solutions then become two weak solutions defined on $[0,T_1]$ and $[T_1,T_2]$, respectively.
Then Lemma~\ref{lem:extend} implies that if $\widehat{\Omega}$ is defined by \eqref{ext-omega}, which is given by
\[
\left\{(x,t)\in\R\times(0,T_2)\ \Bigg|\ \begin{array}{l} x_1(t)<x<x_2(t),\ x_3(t)<x<x_4(t)\ \mbox{ for $0<t< T_1$,}\\
              x_1(t)<x<x_4(t)\ \mbox{ for $T_1\le t< T_2$}\end{array}
\right\},
\]
and $\widehat{v}$ is defined by \eqref{ext-v},
then $(\widehat{\Omega},\widehat{v})$ is a weak solution of \eqref{SLP} for $0\leq t\leq T_2$.

%%%%%%%%%%%%%%%%%%%%%%%%%%%%%%%%%%%%%%%%%%%%%%%%%%

%%%%%%%%%%%%%%%%%%%%%%%%%%%%%%%%%%%%%%%%%%%%%%%%%%

We are ready to show the existence of global weak solutions of \eqref{SLPini}.

%%%%%%%%%%%%%%%%%%%%%%%%%%%%%%%%%%%%%%%%%%%%%%%%%%
\begin{proposition}\label{prop:global weak sol}
%Assume {\bf(H1)} and {\bf(H2)}. Then
There is a global in time
weak solution to problem \eqref{SLPini}.
\end{proposition}
%%%%%%%%%%%%%%%%%%%%%%%%%%%%%%%%%%%%%%%%%%%%%%%%%%
%%%%%%%%%%%%%%%%%%%%%%%%%%%%%%%%%%%%%%%%%%%%%%%%%%
\begin{proof}
By Proposition~\ref{prop:classical sol},
the classical solution $(\Omega,v)$ exists for $t\in[0,T_A)$, where $T_A$ is the annihilation time.
If $T_A=\infty$, then the proof is done.
If $T_A<\infty$,  we define
$(\widetilde{\Omega}, \widetilde{v})$ as in Proposition~\ref{prop:cw} and apply Proposition~\ref{prop:cw},
$(\widetilde{\Omega}, \widetilde{v})$ becomes a weak solution for $0\leq t\leq T_A$. In particular,
{\bf(H1)} and {\bf(H2)} holds with replacing $(\Omega_0,v_0)$ by $(\widetilde{\Omega}(T_A),\widetilde{v}(x,T_A))$.
Then, using Proposition~\ref{prop:classical sol} again,
there exists a unique classical solution $(\Omega_2,v_2)$ for $T_A\leq t<T_2$ for some $T_2>T_A$ with initial data
\[\Omega_2(T_A):=\widetilde{\Omega}(T_A),\quad v_2(x,T_A):=\widetilde{v}(x,T_A).\]
By Proposition~\ref{prop:cw},
$(\Omega_2,v_2)$ can be extended to a weak solution for $T_A\leq t<T_2$ (still denoted by $(\Omega_2,v_2)$).

Next, we define $(\widehat{\Omega}, \widehat{v})$ as in Lemma~\ref{lem:extend} with
$(\Omega_1,v_1):=(\widetilde{\Omega}, \widetilde{v})$.
Then it follows from Lemma~\ref{lem:extend} that $(\widehat{\Omega},\widehat{v})$ is a weak solution for $t\in[0, T_2)$.
If $T_2=\infty$, then Proposition~\ref{prop:global weak sol} follows. Otherwise, $T_2<\infty$ implies that
$T_2$ is the second annihilation time (Proposition~\ref{prop:classical sol}).
We can repeat the above process to extend the weak solution.
Because of {\bf(C2)}, the number of interfaces cannot increase in time, which
implies that the annihilation only occurs finitely many times.
Therefore, by repeating the above process finitely many times, we thus find a globally in time weak solution by gluing
classical solutions. This completes the proof.
\end{proof}
%%%%%%%%%%%%%%%%%%%%%%%%%%%%%%%%%%%%%%%%%%%%%%%%%%

%%%%%%%%%%%%%%%%%%%%%%%%%%%%%%%%%%%%%%%%%%%%%%%%%%
Next, we deal with the uniqueness of the weak solutions. The uniqueness result is given as follows:

\begin{proposition}\label{prop:unique weak}
%Assume {\bf(H1)} and {\bf(H2)}.
Suppose that $(\Omega_1,v_1)$ and $(\Omega_2,v_2)$ are two weak solutions. %{\pur satisfying {\bf(H1)}-{\bf(H3)}} for $0\le t \le T$.
Then
$\Omega_1=\Omega_2$ and $v_1=v_2$.
\end{proposition}
%%%%%%%%%%%%%%%%%%%%%%%%%%%%%%%%%%%%%%%%%%%%%%%%%%%%%%%%%%%%%%%%%%%%%%%%%%%%%%%%%%%%%%%%%%%%%%%%%%%%%%%%%%%%%%%%%5

 The proof of Proposition~\ref{prop:unique weak} is quite involved. We need prepare several lemmas.
The first two lemmas show that any interface of the weak solutions can be locally viewed as a smooth function of $t$.

\begin{lemma}\label{lem:graph1}
Let $(\Omega,v)$ be a weak solution of \eqref{SLPini} for $t\in[0,T]$.
Also, assume that $(x_0,t_0)\in \partial\Omega(t_0)$ for some $t_0\in(0,T)$ such that
\begin{itemize}
\item[{\rm (a)}]  $W(v(x_0,t_0))>0$,
\item[{\rm (b)}]  $(x_0- \varepsilon,t_0)\in \Omega(t_0)$ $(\mbox{resp. }  (x_0+\varepsilon,t_0)\in \Omega(t_0))$ for any sufficiently small $\varepsilon>0$.
\end{itemize}
Then, one of the following cases holds\,{\rm :}
\begin{itemize}
\item[{\rm (i)}]
there are positive constants $\delta$, $\ep$ and a function $x(\cdot)\in C^1((t_0-\delta,t_0+\delta))$ such that $x'(t_0)>0$ $(\mbox{resp. } x'(t_0)<0)$ and
\[
\{(x(t),t)\ |\ t\in (t_0-\delta,t_0+\delta)\}=\partial\Omega\cap\Big[(x_0-\ep,x_0+\ep)\times (t_0-\delta,t_0+\delta)\Big],
\]
\item[{\rm (ii)}]
there are positive constants $\delta$, $\ep$ and two functions $x_i(\cdot)\in C^1((t_0-\delta,t_0])$ $(i=1,2)$ such that
$x_i(t_0)=x_0$,
\beaa
&&\lim_{t\to t_0^-}x_1(t)=\lim_{t\to t_0^-}x_2(t)=x_0,\quad \lim_{t\to t_0^-}x_1'(t)> 0,\quad \lim_{t\to t_0^-}x_2'(t)< 0,\\
&&\bigcup_{i=1}^2\{(x_i(t),t)\ |\ t_0-\delta<t\le t_0\}=
\partial\Omega\cap\Big[(x_0-\ep,x_0+\ep)\times (t_0-\delta,t_0]
\Big].
\eeaa
Namely, $(x_0,t_0)$ is an intersection point of two interfaces.
\end{itemize}
\end{lemma}
\begin{proof}
Since $\partial\Omega$ is Lipschitz continuous (from $(\Omega,v)\in X_T$), there exists a Lipschitz continuous function
(upon relabeling and reorienting the coordinates axis if necessary) whose graph passes through $(x_0,t_0)$.
Also, by {\bf(C2)}, no new interface can generate.
Hence, by taking $\varepsilon>0$ and $\delta>0$ sufficiently small and defining
\beaa
D_{\varepsilon,\delta}:=(x_0-\varepsilon,x_0+\varepsilon)\times (t_0-\delta,t_0+\delta),
\eeaa
we may assume that
$\partial\Omega \cap D_{\varepsilon,\delta}$
contains only one Lipschitz curve passing through $(x_0,t_0)$.
Because of (a) and the continuity of $W$,
we may also assume (if necessary, we may choose $\ep$ and $\delta$ smaller)
\bea\label{v-avlue}
W(v(x,t))>0\quad \mbox{on $D_{\varepsilon,\delta}$}.
\eea

{\bf Claim 1:} we show
\be\label{weak-area}
|\Omega(t_2)\cap (\xi_1,\xi_2)|>|\Omega(t_1)\cap (\xi_1,\xi_2)|
\ee
for all $x_0-\ep<\xi_1<\xi_2<x_0+\ep$ and $t-\delta<t_1<t_2<t_0+\delta$.
Here we recall that $\Omega(t)\cap (\xi_1,\xi_2)=(\xi_1,x(t))$.

By Remark \ref{rem:weak-restricted}, we get
\bea\label{eq:weak-area-org}
\left[\int_{\xi_1}^{\xi_2}{\bf 1}_{\Omega(t)}dx\right]_{t_1}^{t_2}
&=&\int_{\partial \Omega\cap ((\xi_1,\xi_2)\times(t_1,t_2))}W(v) |n_1|d\sigma%
\eea
for $x_0-\ep<\xi_1<\xi_2<x_0+\ep$ and $t_0-\delta<t_1<t_2<t_0+\delta$.
If $\partial \Omega\cap ((\xi_1,\xi_2)\times(t_1,t_2))\ne \emptyset$,
by \eqref{v-avlue}, we have either the right hand side is zero if $|n_1|\equiv0$ or
the right hand side is positive if $|n_1|\not\equiv0$.
If the former case happens, we have
\beaa
|\Omega(t_2)\cap (\xi_1,\xi_2)|=|\Omega(t_1)\cap (\xi_1,\xi_2)|
\eeaa
for $x_0-\ep<\xi_1<\xi_2<x_0+\ep$ and $t_0-\delta<t_1<t_2<t_0+\delta$.
%{ This implies that Lipschitz curve passing through $(x_0,t_0)$ from a vertical line locally.??}
This contradicts with $|n_1|\equiv0$. Hence the latter case must hold
and then {\bf Claim 1} is completed.

{\bf Claim 2:} we show the Lipschitz curve can be represented as a Lipschitz function of $x$ locally. Namely, there exists
a Lipschitz $\varphi(\cdot)$ such that
$t=\varphi(x)$ for $x\in(x_0-\varepsilon,x_0+\varepsilon)$ with $t_0=\varphi(x_0)$.

For contradiction,
suppose that the Lipschitz curve passes through two points $(x_3,t_3)$ and $(x_3,t_4)$ in $(x_0-\ep,x_0+\ep)\times I_{\delta}$.
If $x_3\leq x_0$, taking $\xi_1=x_3-\kappa$ for some sufficiently small $\kappa>0$, $\xi_2=x_3$, $t_1=t_3$ and $t_2=t_4$ it follows from (b) that
\beaa
|\Omega(t_2)\cap (\xi_1,\xi_2)|=|\Omega(t_1)\cap (\xi_1,\xi_2)|,
\eeaa
which contradicts \eqref{weak-area}. Similarly, if $x_3\geq x_0$, we can reach a contradiction. Hence {\bf Claim 2} is completed.

We now complete the proof of this lemma. First  we observe that $t=\phi(\cdot)$ cannot have a local minimum point on $(x_0-\ep,x_0+\ep)$. Otherwise,
by {\bf(C2)} we reach a contradiction immediately. It follows that one of the following cases must occur:
\begin{itemize}
\item[(1)] $t=\phi(\cdot)$ is monotone decreasing on $(x_0-\ep,x_0+\ep)$.
\item[(2)] $t=\phi(\cdot)$ is monotone increasing on $(x_0-\ep,x_0+\ep)$.
\item[(3)] $t=\phi(\cdot)$ is monotone increasing on $(x_0-\ep,\eta)$ and is monotone decreasing on $(\eta,x_0+\ep)$ for some $\eta\in (x_0-\ep,x_0+\ep)$.
\end{itemize}

We shall show that (1) cannot occur. Indeed, from (b) and \eqref{eq:weak-area-org} with $t_i:=\phi(\xi_i)$ for $i=1,2$, we obtain
\beaa
0=|\Omega(t_2)\cap (\xi_2,\xi_1)|>|\Omega(t_1)\cap (\xi_2,\xi_1)|,
\eeaa
which is impossible. This shows that (1) never occurs.

For (2), we furthermore show that $t=\phi(\cdot)$ is strictly increasing.
If it is not true, then its graph contains a flat piece, say $\phi(x)=\kappa$ on $(p_1,p_2)$ for some $\kappa>0$ and $p_i\in(x_0-\ep,x_0+\ep)$.
Then by  {\bf(C2)}, we reach a contradiction. Hence $t=\phi(\cdot)$ is strictly increasing. So its inverse function $x=x(t)$ is well defined.
Moreover,
by \eqref{eq:weak-area-org} with $\xi_i=x(t_i)$ for $i=1,2$, we have
\be\label{eq:x}
x(t_2)-x(t_1)=\int_{t_1}^{t_2}W(v(x(t),t))dt.
\ee
Differentiating \eqref{eq:x} in $t_2$ and using \eqref{v-avlue}, { we obtain} the conclusion (i) of Lemma~\ref{lem:graph1}.

For (3), if $x_0<\eta$, then we can shrink $\ep>0$ and $\delta>0$ sufficiently small and reduces this case into case (2).
Then by the same process we can obtain the conclusion (i) of Lemma~\ref{lem:graph1}. If $x_0=\eta$, then following the process of (2)
we see that $t=\phi(\cdot)$ is strictly increasing (resp. decreasing) on $(x_0-\ep,x_0]$ (resp. $[x_0,x_0+\ep)$). Hence
there exist two continuous functions $x_1(t)$ and $x_2(t)$ such that $x_i(t_0)=x_0$ for $i=1,2$. Using \eqref{eq:weak-area-org}, we can
obtain the conclusion (ii) of Lemma~\ref{lem:graph1}. This completes the proof.
\end{proof}

Similar result holds for $W(v(x_0,t_0))<0$. We state the result as follows without repeating a proof.

\begin{lemma}\label{lem:graph2}
Let $(\Omega,v)$ be a weak solution of \eqref{SLPini} for $t\in[0,T]$.
Also, assume that $(x_0,t_0)\in \partial\Omega(t_0)$ for some $t_0\in(0,T)$ such that
\begin{itemize}
\item[{\rm (a)}]  $W(v(x_0,t_0))<0$,
\item[{\rm (b)}]  $(x_0- \varepsilon,t_0)\in \Omega(t_0)$ $(\mbox{resp. }  (x_0+\varepsilon,t_0)\in \Omega(t_0))$ for any sufficiently small $\varepsilon>0$.
\end{itemize}
Then one of the following cases holds:
\begin{itemize}
\item[{\rm (i)}]
there are positive constants $\delta$, $\ep$ and a function $x(\cdot)\in C^1((t_0-\delta,t_0+\delta))$ such that $x'(t_0)<0$ $(\mbox{resp. } x'(t_0)>0)$ and
\[
\{(x(t),t)\ |\ t\in (t_0-\delta,t_0+\delta)\}= \partial\Omega\cap\Big[(x_0-\ep,x_0+\ep)\times (t_0-\delta,t_0+\delta)\Big],
\]
\item[{\rm (ii)}]
there are positive constants $\delta$, $\ep$ and two functions $x_i(\cdot)\in C^1((t_0-\delta,t_0])$ $(i=1,2)$ such that
$x_i(t_0)=x_0$,
\beaa
&&\lim_{t\to t_0^-}x_1(t)=\lim_{t\to t_0^-}x_2(t)=x_0,\quad \lim_{t\to t_0^-}x_1'(t)> 0,\quad \lim_{t\to t_0^-}x_2'(t)< 0,\\
&&\bigcup_{i=1}^2\{(x_i(t),t)\ |\ t_0-\delta<t\le t_0\}=
\partial\Omega\cap\Big[(x_0-\ep,x_0+\ep)\times (t_0-\delta,t_0]
\Big].
\eeaa
Namely, $(x_0,t_0)$ is a intersection point of two interfaces.
\end{itemize}
\end{lemma}

\begin{remark}\label{graph-t0}
(i) We remark that Lemmas~\ref{lem:graph1} and ~\ref{lem:graph2} still hold for $t_0=0$. Indeed, following the same process in the proof
therein but replace $(t_0-\delta,t_0+\delta)$ by $[0,\delta)$ we can show that conclusion (i) with $(t_0-\delta,t_0+\delta)$ replaced by $[0,\delta)$
in Lemmas~\ref{lem:graph1} and ~\ref{lem:graph2} always occurs when $t_0=0$.

(ii) Lemmas~\ref{lem:graph1} and ~\ref{lem:graph2} show that
if $(\Omega,v)$ is a weak solution for $t\in[t_0,t_1]$, then there exists an integer $N$ such that $(\Omega,v)$ becomes a classical solution
for $t\in[\tau_i, \tau_{i+1})$ for some $i=0,1,...,N$ with $\tau_0=t_0$ and $\tau_N=t_1$, where $\tau_i$ is an annihilation time for $1\leq i\leq N-1$.
\end{remark}

Suppose that $(\Omega_1,v_1)$ and $(\Omega_2,v_2)$ are two weak solutions of \eqref{SLPini} for $t\in[0,T]$.
The following lemma shows that $v_1=v_2$ over a rectangle
lying in either $\Omega_1\cap\Omega_2$ or
$\Omega_1^c\cap\Omega_2^c$, and its bottom edge lies on the $x$-axis.

%%%%%%%%%%%%%%%%%%%%%%%%%%%%%%%%%%%%%%%%%%%%%%%%%%
\begin{lemma}\label{lem:Q}
Let $(\Omega_1,v_1)$ and $(\Omega_2,v_2)$ be two weak solutions of \eqref{SLPini} {for $t\in[0,T]$}.
Then $v_1=v_2$ in $\overline Q$, where
\beaa
Q:=\Big\{(x,t)\ |\ \mbox{$\exists$ $\varepsilon>0$ such that $[x-\varepsilon,x+\varepsilon]\times[0,t]$ lies in either $\Omega_1\cap\Omega_2$ or $\Omega_1^c\cap\Omega_2^c$}\Big\}.
\eeaa
\end{lemma}
%%%%%%%%%%%%%%%%%%%%%%%%%%%%%%%%%%%%%%%%%%%%%%%%%%
%%%%%%%%%%%%%%%%%%%%%%%%%%%%%%%%%%%%%%%%%%%%%%%%%%
\begin{proof}
Let $(x_0,t_0)\in Q$. Let $L_0>0$ (resp. $L_1>0$) be the Lipschitz constant for $g(0,\cdot)$ (resp. $g(1,\cdot)$).
First, we choose $0<t_1<t_0$ such that
\be\label{small time}
\max\{L_0,L_1\}t_1< \sqrt 2.
\ee
Set $z:={\bf 1}_{\Omega_1}-{\bf 1}_{\Omega_2}$ and $w:=v_1-v_2$.
%There are a positive constant $\ep$ and an interval $I_{\ep}=[x_0-\ep,x_0+\ep]$ such that $I_{\ep}\times \{t_0\}\in Q$.
{Since $(x_0,t_0)\in Q$, we can find an interval $I_{\ep}=[x_0-\ep,x_0+\ep]$ such that $I_{\ep}\times [0,t_0]$ lies in either $\Omega_1\cap\Omega_2$ or $\Omega_1^c\cap\Omega_2^c$.
It follows that  $z=0$ in $I_{\ep}\times [0,t_0]$.
In particular, $z=0$ in $I_{\ep}\times [0,t_1]$.}
%Then $z=0$ in $I_{\ep}\times (0,t_0)$.

By Lemma \ref{lem:wk-t}, $(\Omega_1,v_1)$ and $(\Omega_2,v_2)$ are two weak solutions of \eqref{SLPini} %in
{for $0<t<t_1$}.
By definition of {weak solutions} and choosing a test function $\psi$ satisfying
\[
\psi:={\bf 1}_{I_{\ep}}(x)\int_t^{t_1} w^+(x,s)ds,\quad w^+:=\max\{w,0\},
\]
we have
\beaa
0&=&\left[\int_{\R}w\psi dx\right]_0^{t_1}\\
&=&
\int_0^{t_1}\int_{\R}\Big(w\psi_t+(g({\bf 1}_{\Omega_1(t)},v_1)-g({\bf 1}_{\Omega_2(t)},v_2))\psi\Big )dxdt\\
&\leq&-\int_0^{t_1}\int_{x_0-\ep}^{x_0+\ep}(w^+)^2dxdt+\max\{L_0,L_1\}\int_0^{t_1}\int_{x_0-\ep}^{x_0+\ep}w\Big(\int_t^{t_1}w^+(x,s)ds\Big) dxdt\\
&\leq&-\int_{x_0-\ep}^{x_0+\ep}\int_0^{t_1}(w^+)^2dtdx+ \frac1{\sqrt{2}}\max\{L_0,L_1\}t_1\int_{x_0-\ep}^{x_0+\ep}\int_0^{t_1}(w^+)^2 dtdx,
\eeaa
where we have used the H\"{o}lder inequality to obtain the last inequality.

{By \eqref{small time}, from the above estimate we must have $w^+=0$ and so $w\le 0$ in $I_{\ep}\times[0,t_1]$.
Similarly, we have $w\ge 0$ and so $w=0$ in $I_{\ep}\times[0,t_1]$. The above process can apply to derive
that $w=0$ in $I_{\ep}\times[t_1,\min\{2t_1,t_0\}]$. Using a bootstrap argument, we can obtain that $w=0$ in $I_{\ep}\times[0,t_0]$.
This completes the proof.}
\end{proof}
%%%%%%%%%%%%%%%%%%%%%%%%%%%%%%%%%%%%%%%%%%%%%%%%%%

\begin{remark}\label{rk:thin rectangle}
From the proof of Lemma~\ref{lem:Q}, we see that the conclusion still hold when the interval $[x-\varepsilon,x+\varepsilon]$ in $Q$ is replaced by
$[x,x+\varepsilon]$ or $[x-\varepsilon,x]$.
\end{remark}

\medskip

\begin{lemma}\label{lem:key}
Let $(\Omega_1,v_1)$ and $(\Omega_2,v_2)$ be two weak solutions of \eqref{SLPini} for $t\in[0,T]$.
Suppose that there exist a rectangle $D_T:=[\xi_1,\xi_2]\times [0,T]$ and continuous functions $y_k(t)$, $k=1,2$,
such that
\begin{itemize}
\item[{\rm (i)}] $\xi_1<y_1(t)\le y_2(t)<\xi_2$ for $0\le t\le T$ and $y_1(0)=y_2(0)${\rm ;}
\item[{\rm (ii)}] either $\Omega_k\cap D_T=\{(x,t)\ |\ \xi_1<x<y_k(t),\ 0\leq t\leq T \}$ for $k=1,2$ or $\Omega_k\cap D_T=\{(x,t)\ |\ y_k(t)<x<\xi_2,\ 0\leq t\leq T \}$ for $k=1,2$.
\end{itemize}
Then $\Omega_1\cap D_T=\Omega_2\cap D_T$ and $v_1=v_2$ in $D_T$.%\overline D.
\end{lemma}
\begin{proof}
Since the proof is similar, we only consider the case $\Omega_k\cap D_T=\{(x,t)\ |\ \xi_1<x<y_k(t),\ 0\leq t\leq T \}$ for $k=1,2$. {Because of assumptions (i) and (ii), we can apply Lemma \ref{lem:Q} (and Remark~\ref{rk:thin rectangle}) so that $v_1(\xi_j,t)=v_2(\xi_j,t)$ for $0<t<T$ and $j=1,2$
and $v_1(y_2(t),t)=v_2(y_2(t),t)$ for $0<t<T$.
}
%$v_1(y_2(t),t)=v_2(y_2(t),t)$ for $0<t<T$.

Taking test functions ${\bf 1}_{[\xi_1,\xi_2]}\varphi$ and ${\bf 1}_{[\xi_1,\xi_2]}\psi$ implies
\beaa
\left[\int_{\xi_1}^{y_k(t)}\varphi dx\right]_0^T
&=&
\int_0^T\int_{\xi_1}^{y_{k}(t)}\varphi_tdxdt
+\int_0^TW(v_k(y_k(t),t))\varphi(y_k(t),t) dt,\\
\left[\int_{\xi_1}^{\xi_2}v_k\psi dx\right]_0^T
&=&
\int_0^T\int_{\xi_1}^{\xi_2}\Big(v_k\psi_t+g({\bf 1}_{\Omega_k(t)},v_k)\psi\Big )dxdt
\eeaa
for $k=1,2$.
Set $z:={\bf 1}_{\Omega_2}-{\bf 1}_{\Omega_1}$ and $w:=v_2-v_1$.
Subtracting each equality for $k=1,2$, we get
\bea
\left[\int_{y_1(t)}^{y_2(t)}\varphi dx\right]_0^T
&=&\int_0^T\int_{y_1(t)}^{y_2(t)}\varphi_tdxdt\label{eq-w1-1}\\
&&+\int_0^T\Big(W(v_2(y_2(t),t))\varphi(y_2(t),t)-W(v_1(y_1(t),t))\varphi(y_1(t),t)\Big) dt,\nonumber\\
\left[\int_{\xi_1}^{\xi_2}w\psi dx\right]_0^T
&=&\int_0^T\int_{\xi_1}^{\xi_2}\Big(w\psi_t+(g({\bf 1}_{\Omega_2(t)},v_2)-g({\bf 1}_{\Omega_1(t)},v_1))\psi\Big )dxdt.\label{eq-w2-1}
\eea

First, we show that $v_1\le v_2$ in $D_T$.
%\[
%v_1(y_2(t),t)\le v_2(y_2(t),t)
%\]
%for $0\le t\le T$.
Plugging
\[
\psi(x,t)=
\begin{cases}
0\quad &\mbox{if $w(x,t)\ge 0$,}\\
-1 &\mbox{if $w(x,t)< 0$,}
\end{cases}
\]
we get from \eqref{eq-w2-1} that
\beaa
\int_{\xi_1}^{\xi_2}w^-(x,T)dx=
\int_0^T\int_{\xi_1}^{\xi_2}\Big(g({\bf 1}_{\Omega_2(t)},v_2)-g({\bf 1}_{\Omega_1(t)},v_1)\Big )\psi dxdt,
\eeaa
where $w^-:=\max\{-w,0\}\ge 0$.
Also, we see from assumptions (i) and (ii) that
\beaa
\Big(g({\bf 1}_{\Omega_2(t)},v_2)-g({\bf 1}_{\Omega_1(t)},v_2)\Big )\psi\leq0\quad \mbox{in $D_T$}.
\eeaa
This implies
\beaa
\int_{\xi_1}^{\xi_2}w^-(x,T)dx%\psi dx
&= &
\int_0^T\int_{\xi_1}^{\xi_2}\Big(g({\bf 1}_{\Omega_2(t)},v_2)-g({\bf 1}_{\Omega_1(t)},v_2)\Big )\psi dxdt\\
&&+\int_0^T\int_{\xi_1}^{\xi_2}\Big(g({\bf 1}_{\Omega_1(t)},v_2)-g({\bf 1}_{\Omega_1(t)},v_1)\Big )\psi dxdt\\
%&=&
%\int_0^T\int_{\xi_1}^{\xi_2}\int_0^1g_v({\bf 1}_{\Omega_1(t)},\theta v_2+(1-\theta)v_1)w^- d\theta dxdt\\
&\le &C
\int_0^T\int_{\xi_1}^{\xi_2} w^- dxdt.
\eeaa
%where we used that $v$ is bounded in $\overline D$.
The Gronwall inequality implies that $w^-\equiv  0$.
Namely, $v_1\le v_2$ in $D_T$.

Next, we show that $\Omega_1\cap D_T=\Omega_2\cap D_T$.
Plugging
\[
\varphi(t)=\int_t^T(y_2(s)-y_1(s))ds\ge 0,\quad \psi(x,t)={\bf 1}_{[\xi_1,\xi_2]}(x)\int_t^Tw(x,s)ds
\]
into \eqref{eq-w1-1} and \eqref{eq-w2-1} implies
\bea
0&=&
-\int_0^T(y_2-y_1)^2dt
-b\int_0^T\Big(v_2(y_2(t),t)-v_1(y_1(t),t)\Big)\varphi(t) dt,\label{unique-weak1}\\
0&=&
-\int_0^T\int_{\xi_1}^{\xi_2}w^2dxdt+\int_0^T\int_{\xi_1}^{\xi_2}\Big(g({\bf 1}_{\Omega_2(t)},v_2)-g({\bf 1}_{\Omega_1(t)},v_1)\Big )\psi dxdt.\label{unique-weak2}
\eea
{Recall that $v_1(y_2(t),t)=v_2(y_2(t),t)$ for $0<t<T$.
%Then the first equation implies
Then from \eqref{unique-weak1} we have}
\beaa
0&\le &
-\int_0^T(y_2-y_1)^2dt
{- b\int_0^T\Big(v_2(y_2(t),t)-v_1(y_2(t),t)\Big)\varphi(t) dt}\\
&&-b\int_0^T\Big(v_1(y_2(t),t)-v_1(y_1(t),t)\Big)\varphi(t) dt\\
&\le&-\int_0^T(y_2-y_1)^2dt
+b\,\varphi(0)\int_0^T |v_1|_{\rm Lip}(y_2(t)-y_1(t))dt\, \\
&\le&-\int_0^T(y_2-y_1)^2dt
+b |v_1|_{\rm Lip} \varphi(0)^2\\
&\le&-\Big(1-bT|v_1|_{\rm Lip}\Big)\int_0^T(y_2-y_1)^2dt,
\eeaa
by using the Schwarz inequality.
We conclude that %$\Omega_1\cap ([\xi_1,\xi_2]\times[0,T])=\Omega_2\cap ([\xi_1,\xi_2]\times[0,T])$ for small $T>0$.
{$\Omega_1\cap D_T=\Omega_2\cap D_T$ for small $T>0$}.
Thus {from \eqref{unique-weak2}}  we have
\beaa
0&=&
-\int_0^T\int_{\xi_1}^{\xi_2}w^2dxdt+\int_0^T\int_{\xi_1}^{\xi_2}\Big(g({\bf 1}_{\Omega_1(t)},v_2)-g({\bf 1}_{\Omega_1(t)},v_1)\Big )\psi dxdt.
\eeaa
From this identity, by a similar argument as above leads that
\beaa
0&\le&-(1-T|g_v|_{L^\infty})\int_0^T\int_{\xi_1}^{\xi_2}|w|^2dxdt.
\eeaa
Thus we get $v_1=v_2$ in $[\xi_1,\xi_2]\times[0,T]$ for small $T>0$.

Therefore, by a bootstrap argument as in Lemma~\ref{lem:Q}, we can complete the proof.
\end{proof}

We are ready to show Proposition~\ref{prop:unique weak}.

\begin{proof}[Proof of Proposition~\ref{prop:unique weak}]
First, we claim that $\Omega_1=\Omega_2$ for $0<t<\tau$ where $\tau$ is small enough. Given any $x_0\in\p\Omega(0)$, let us consider the case of Lemma~\ref{lem:graph1} (i) or Lemma~\ref{lem:graph2} (i). Then together with Remark~\ref{graph-t0},
there are positive constants $\delta_i$, $\varepsilon_i$ and functions $x_i(\cdot)\in C^1([0, \delta_i))$, $i=1,2$ such that
\[
\{(x_i(t),t)\ |\ t\in [0, \delta_i)\}= \partial\Omega_i\cap\Big[(x_0-\ep_i,x_0+\ep_i)\times [0, \delta_i)\Big],\quad i=1,2.
\]
Taking $\xi_1=x_0-\ep, \xi_2=x_0+\ep$ with $\ep=\max\{\ep_1, \ep_2\}$ and $T=\min\{\delta_1, \delta_2\}$. In view of Lemma~\ref{lem:key} we have $x_1(t)=x_2(t)$ for $t\in [0, T]$. Since $x_0$ is given arbitrarily and $\Omega(0)$ consists of a finite number of bounded intervals, there exists a $\tau$ small enough such that $\Omega_1(t)=\Omega_2(t)$ for $0\le t\le\tau$.

Next, for $i=1,2$, we set $T_i$ as an annihilation time of $\Omega_i$. Without loss of generality, we assume $T_1\le T_2$. By a bootstrap argument as in Lemma~\ref{lem:Q}, we have $\Omega_1(t)=\Omega_2(t)$ for $0\le t\le\tau$ for any $\tau\le T_1$. By the definition of the annihilation time, we know that $T_1=T_2$. Applying Lemma~\ref{lem:graph1} (ii), Lemma~\ref{lem:graph2}  (ii) and Lemma~\ref{lem:key}, we obtain that $\Omega_1(t)=\Omega_2(t)$ for $0\le t\le T_1$. Again, by a bootstrap argument as in Lemma~\ref{lem:Q}, we have $\Omega_1=\Omega_2$.

Finally, set $\Omega_1=\Omega_2=\Omega$. If $(x_0, t_0)\in \partial\Omega$, we have $v_1(x_0, t_0)=v_2(x_0, t_0)$ by { Lemmas~\ref{lem:graph1}, ~\ref{lem:graph2} and ~\ref{lem:key}}; if $(x_0, t_0)\notin \partial\Omega$, we obtain $v_1(x_0, t_0)=v_2(x_0, t_0)$ by Lemma~\ref{lem:Q}. The proof of this proposition is done.
\end{proof}

\medskip
We end this section with the proof of Theorem~\ref{thm:global weak sol}.

\begin{proof}[Proof of Theorem~\ref{thm:global weak sol}]
Theorem~\ref{thm:global weak sol} follows from Proposition~\ref{prop:global weak sol} and Proposition~\ref{prop:unique weak}.
\end{proof}

\medskip

%%%%%%%%%%%%%%%%%%%%%%%%%%%%%%

\medskip

\noindent{\bf Acknowledgments}
The first author is partly supported by the Ministry of Science and Technology of Taiwan under the grant MOST 105-2115-M-032-005-MY3, MOST 105-2811-M-032-007, MOST 106-2811-M-032-009 and MOST 107-2811-M-032-502.
The second author was partially supported by JSPS KAKENHI Grant Numbers { JP16KT0022 and JP20H01816}. The second author would like to thank
the Mathematics Division of NCTS (Taipei Office) for the warm hospitality and the support of the second author's visit to Taiwan.
The third author is partly supported from the
Young Scholar Fellowship Program by Ministry of Science and Technology (MOST) in Taiwan, under
Grants MOST 108-2636-M-009-009, MOST 109-2636-M-009-008 and MOST 110-2636-M-009-006.
He also thanks Meiji University for the hospitality during his visit.

\setcounter{equation}{0}
\setcounter{theorem}{0}
\section*{Appendix}

\renewcommand{\thesection}{A}

%%%%%%%%%%%%%%%%%%%%%%%%%%%%%%%%%%%%%%%%%%%%%%%%%%
\begin{proof}[Proof of Lemma \ref{lem:lipschitz-appendix}]
We divide $[A,B]\times[0,\tau]$ into
\beaa
\Big(\bigcup_{j=1}^{2m}{x}_j([0,\tau])\times[0,\tau]\Big),\quad  [A,B]\times[0,\tau]\setminus\Big(\bigcup_{j=1}^{2m}{x}_j([0,\tau])\times[0,\tau]\Big).
\eeaa
By \eqref{non-overlapping}, we see that $I_k:={x}_k([0,\tau])\times[0,\tau]$ contains exactly one interface $x={x}_k(t)$, $t\in[0,\tau]$.
This allows us to represent $v$ in terms of functions in \eqref{G func}. Indeed,
by some simple computations, we have the following:
\begin{itemize}
\item[(i)] For $j\in\{1,2,...,m\}$, if ${x}_{2j}(\cdot)$ is increasing (resp. ${x}_{2j-1}(\cdot)$ is decreasing), then
\beaa
v(x,t)=\begin{cases}
   G_0\left(G_0^{-1}(v_0(x))+t\right),&\ t\leq T_k(x),\ (x,t)\in I_k,\\
   G_1\left(G_1^{-1}(v(x,T_k(x)))+t-T_k(x)\right),&\ t> T_k(x),\ (x,t)\in I_k,
   \end{cases}
\eeaa
where $k=2j$ (resp. $k=2j-1$).

\medskip

\item[(ii)] For $j\in\{1,2,...,m\}$, if ${x}_{2j}(\cdot)$ is decreasing (resp. ${x}_{2j-1}(\cdot)$ is increasing), then
\beaa
v(x,t)=\begin{cases}
   G_1\left(G_1^{-1}(v_0(x))+t\right),&\ t\leq T_k(x),\ (x,t)\in I_k,\\
   G_0\left(G_0^{-1}(v(x,T_k(x)))+t-T_k(x)\right),&\ t> T_k(x),\ (x,t)\in I_k,
   \end{cases}
\eeaa
where $k=2j$ (resp. $k=2j-1$).
\end{itemize}

Let us first deal with (i). For $t\leq T_k(x)$, since
\beaa
v(x,t)=G_0\left(G_0^{-1}(v_0(x))+t\right),
\eeaa
as in deriving \eqref{Lip-outside}, Lemma~\ref{lem:G0 Lip} and the Lipschitz continuity of $v_0$ imply
the Lipschitz continuity of $v$ for $(x,t)\in I_k$ with $t\leq T_k(x)$.
For $t> T_k(x)$, since $T_k$ is strictly monotone (because of \eqref{phi-monotone}), there exist a unique $z\in(x,x_k(\tau))\subset I_k$ such that
$t=T_k(z)$ (see Figure \ref{fig:1}).

%%%%%%%%%%%%%%%%%%%%%%%%%%%%%%%%%%%%%%%%%%%%%%%%%%
\begin{figure}
% Use the relevant command to insert your figure file.
% For example, with the graphicx package use
  \centering\includegraphics[width=0.4\textwidth]{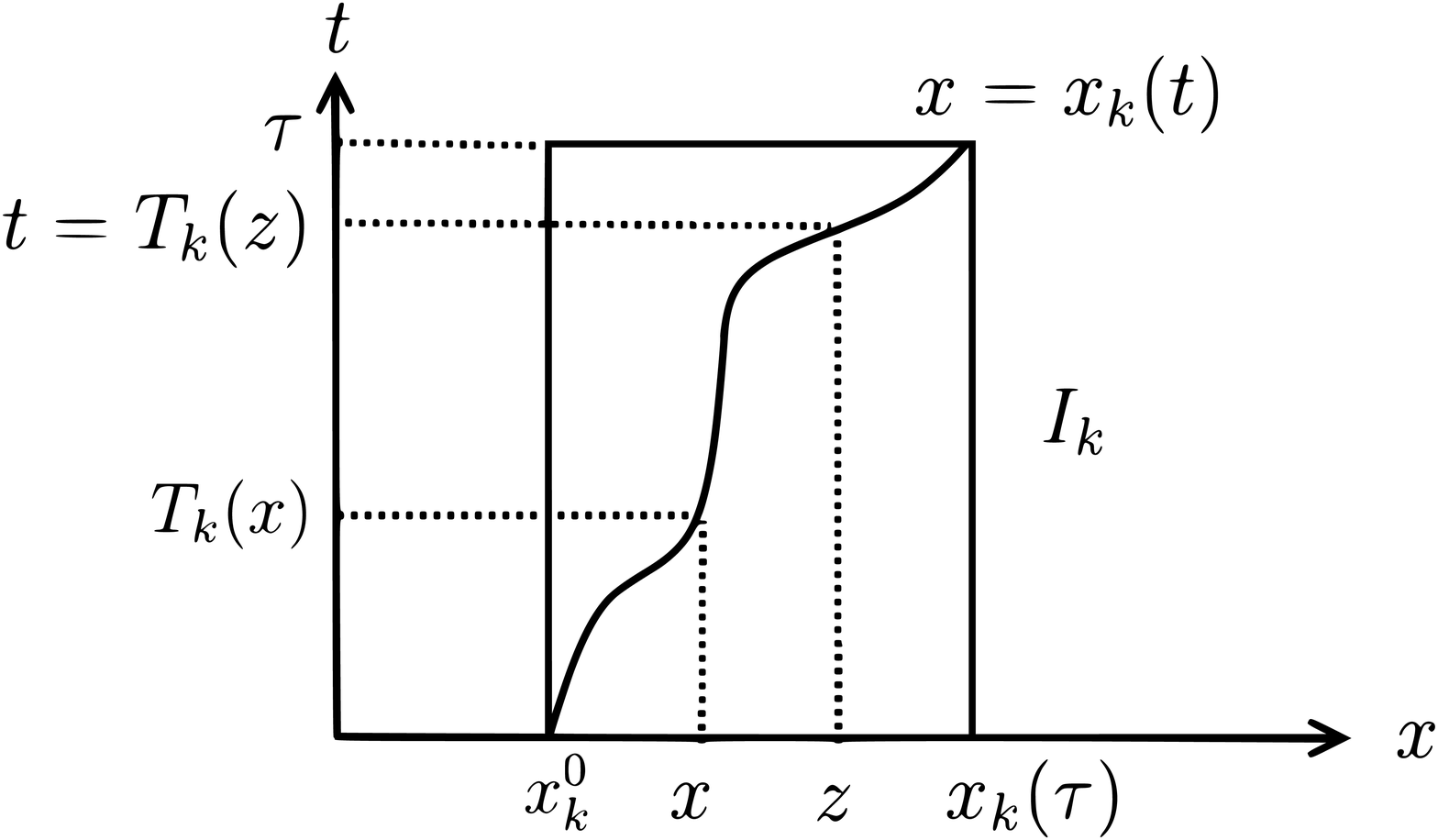}
% figure caption is below the figure
\caption{A diagram of $I_k$.}
\label{fig:1}       % Give a unique label
\end{figure}
%%%%%%%%%%%%%%%%%%%%%%%%%%%%%%%%%%%%%%%%%%%%%%%%%%

Then, if $x<z\le y<x_k(\tau)$, namely, $T_k(x)<t\le T_k(y)$,  we have
\bea
\lefteqn{|v(x,t)-v(y,t)|\label{Lip-inside-est1}}\\
&=&|G_1\left(G_1^{-1}(v(x,T_k(x)))+T_k(z)-T_k(x)\right)-G_0\left(G_0^{-1}(v(y,0))+T_k(z)\right)|\notag\\
&\leq &|G_1\left(G_1^{-1}(v(x,T_k(x)))+T_k(z)-T_k(x)\right)-v(z,T_k(z))|\notag\\
   &&+|G_0\left(G_0^{-1}(v(z,0))+T_k(z)\right)-G_0\left(G_0^{-1}(v(y,0))+T_k(z)\right)|\notag\\
   &{ =:}&J_1+J_2.\notag
\eea
Let us define
\beaa
C_1:=\sup_{(x,t)\in[A,B]\times[0,T]}|v(x,t)|,\quad C_2:=G_1^{-1}(C_1)+T.
\eeaa
Using $v(z,T_k(z))=G_1(G_1^{-1}(v(z,T_k(z))))$ and the mean value theorem, we have
\beaa
J_1&\leq&\|G'_1\|_{L^{\infty}([0,C_2])}|G_1^{-1}(v(x,T_k(x)))-G_1^{-1}(v(z,T_k(z)))+T_k(z)-T_k(x)|\\
&\leq&\|G'_1\|_{L^{\infty}([0,C_2])}\Big(\|(G^{-1}_1)'\|_{L^{\infty}([0,C_1])}|v(x,T_k(x))-v(z,T_k(z))|+|T_k(z)-T_k(x)|\Big).
\eeaa
Note that $(G_1^{-1})'(v)=1/g(1,v)$ and $0<g_1-g_2/g_3\le g(1,v)\le g_1$. Hence $\|G'_1\|_{L^{\infty}([0,C_2])}<\infty$.
Next, using $v(\zeta,T_k(\zeta))=G_0\left(G_0^{-1}(v_0(\zeta))+T_k(\zeta)\right)$ for $\zeta=x,z$ and Lemma~\ref{lem:G0 Lip}, we have
\beaa
J_1&\leq&C_3\Big(|v_0(z)-v_0(x)|+|T_k(z)-T_k(x)|\Big)\\
&\leq&C_3\Big(L_0|x-z|+\frac{1}{\delta}|x-z|\Big)\le C_4|x-y|
\eeaa
for some positive constant $C_3$ and $C_4$. We can also get
%Similarly, we can derive
\beaa
J_2&\leq& C_5|y-z|\le C_5|x-y|
\eeaa
for some positive constant $C_5$ by the same argument to the proof of outside of $D_T$.
Combining the estimates for $J_1$ and $J_2$, from \eqref{Lip-inside-est1}, we see that $v$ is Lipschitz continuous when $T_k(x)<t$ and $T_k(y)\ge t$. For the case where $T_k(x)\le T_k(y)<t$, we directly obtain
\beaa
\lefteqn{|v(x,t)-v(y,t)|}\\
&=&|G_1\left(G_1^{-1}(v(x,T_k(y)))+t-T_k(y)\right)-G_1\left(G_1^{-1}(v(y,T_k(y)))+t-T_k(y)\right)|\notag\\
&\leq &\|G'_1\|_{L^{\infty}([0,C_2])}\|(G^{-1}_1)'\|_{L^{\infty}([0,C_1])}\left|v(x,T_k(y))-v(y,T_k(y))\right|.
\eeaa
Since this reduces to the previous case,
we see that $v$ is Lipschitz continuous on $I_k$ when (i) holds.
%Then the similar argument leads us to show the Lipschitz continuity of $v$.
The similar process is applicable to assert the Lipschitz continuity of $v$ on $I_k$ when (ii) occurs. We omit the detailed proof.
\end{proof}
%%%%%%%%%%%%%%%%%%%%%%%%%%%%%%%%%%%%%%%%%%%%%%%%%%

%%%%%%%%%%%%%%%%%%%%%%%%%%%%%%%%%%%%%%%%%%%%%%%%%%
\begin{proof}[Proof of Lemma \ref{lem:A}]
By Proposition~\ref{prop:local-existence},
there exists a positive constant $\tau_0$ such that problem \eqref{SLPini} has a classical solution for $t\in[0,\tau_0)$ and
$x_k(\cdot)$ is strictly monotone on $[0,\tau_0)$ for each $k=1,...,2m$.

We show that the local in time classical $(\Omega,v)$ is unique.
Let $(\Omega,v)$ and $(\widetilde \Omega,\widetilde v)$ be two local in time classical solutions of \eqref{SLPini}  for $t\in[0,\tau_0)$, where
\beaa
\widetilde \Omega(t)=\bigcup_{j=1}^{m}(\widetilde x_{2j-1}(t),\widetilde x_{2j}(t)).
\eeaa
Moreover, we may assume that $x'_k(\cdot)$ and $\widetilde{x}'_k(\cdot)$ never vanish in $[0,\tau_0)$ by choosing $\tau_0$ sufficiently small.
We shall show that $\Omega=\widetilde \Omega$ and $v=\widetilde v$ for $t\in[0,\tau_0)$.
For this, we first show $x_{2m}(t)=\widetilde{x}_{2m}(t)$ for $t\in[0,\tau_0)$. Without loss of generality,
we may assume that $W(v_0(x^0_{2m}))>0$. i.e.,
$x'_{2m}(\cdot)>0$ and $\widetilde{x}'_{2m}(\cdot)>0$.
Without loss of generality, we may assume that
\bea\label{order}
x_{2m}(t)>\widetilde{x}_{2m}(t),\quad t\in(0,\tau_0).
\eea
In fact, the following proof with slight modifications works for the case that $x_{2m}\geq\widetilde{x}_{2m}$.

Let $T_k(\xi)$ (resp. $\widetilde T_k(\xi)$) be the arrival time of $x_k(t)$ (resp. $\widetilde x_k(t)$) at $\xi$.
For each $\xi\in(x_{2m}^0, \widetilde{x}_{2m}(\tau_0^-))$,
\bea\label{T-eq}
\begin{cases}
T_{2m}'(\xi)=\dps\frac{1}{x_{2m}'(T_{2m}(\xi))}=\dps\frac{1}{a-b v(\xi,T_{2m}(\xi))},\vspace{2mm}\\
\widetilde{T}_{2m}'(\xi)=\dps\frac{1}{\widetilde{x}_{2m}'(\widetilde{T}_{2m}(\xi))}=\dps\frac{1}{a-b \widetilde{v}(\xi,\widetilde{T}_{2m}(\xi))}.
\end{cases}
\eea
By \eqref{order}, we have
\bea\label{order2}
T_{2m}(\xi)<\widetilde{T}_{2m}(\xi),\quad \xi\in[x_{2m}^0,\widetilde{x}_{2m}(\tau_0^-)).
\eea
Note that $\widetilde{v}_t(\xi,t)=g(0,\widetilde{v})<0$ for $t\in[{T}_{2m}(\xi),\widetilde{T}_{2m}(\xi)]$,
and  $\widetilde{v}(\xi,T_{2m}(\xi))={v}(\xi,T_{2m}(\xi))$, we have
\beaa
v(\xi,T_{2m}(\xi))>\,\widetilde{v}(\xi,\widetilde{T}_{2m}(\xi)),\quad \xi\in(x_{2m}^0, \widetilde{x}_{2m}(\tau_0^-)).
\eeaa
Together with \eqref{T-eq},
we see that
\beaa
\widetilde{T}'_{2m}(\xi)<{T}'_{2m}(\xi),\quad \xi\in(x_{2m}^0, \widetilde{x}_{2m}(\tau_0^-)).
\eeaa
By integrating over $[x_{2m}^0,\xi]$, we reach a contradiction with \eqref{order2}. Thus, we must have $x_{2m}(t)=\widetilde{x}_{2m}(t)$ for $t\in[0,\tau_0)$.
The same argument can be applied to prove $x_{j}(t)=\widetilde{x}_{j}(t)$ for $t\in[0,\tau_0)$ for $j=1,...,2m-1$, but
the details are tedious. We safely omit the details here.
Thus, we obtain $\Omega=\widetilde \Omega$. Moreover, $v$ and $\widetilde{v}$ can be represented in terms of
the functions in \eqref{G func} as in the proof of Proposition~\ref{prop:local-existence}. Note that
$\Omega=\widetilde \Omega$ implies that $T_i(x)=\widetilde{T}_i(x)$ for any $x\in\mathbb{R}$ and $i=1,\cdots,2m$. It follows that $v=\widetilde{v}$.
%Let $\widetilde T_k(x)$ be { the} arrival time of $\widetilde x_k(t)$ at $x$.
%If $0<\widetilde T_k(x)\le t$ and the other interfaces do not cross at $x$ in $[0,t]$, then
%\beaa
%\int_0^t|{\bf 1}_{\Omega(s)}(x)-{\bf 1}_{\widetilde \Omega(s)}(x)|ds \le
%|\min\{t,T_k(x)\}-\widetilde T_k(x)|
%\le  \dfrac 1{\delta}\max_{s\in[0,t]}|x_k(s)-\widetilde x_k(s)|,
%\eeaa
%as long as \eqref{phi-monotone} holds, where $\delta>0$ is defined in \eqref{phi-monotone}.
%%\beaa
%%x_k'-\widetilde x_k'=W(v(x_k,t))-W(\widetilde v(\widetilde x_k,t)),\quad v_t-\widetilde v_t=g({\bf 1}_{\Omega(t)},v)-g({\bf 1}_{\widetilde \Omega(t)},\widetilde v)
%%\eeaa
%%
From the above discussion, one can use a bootstrap argument to extend the local in time classical solution uniquely until $x_k'$ vanishes for some $k$
or an annihilation occurs. This completes the proof.
\end{proof}
%%%%%%%%%%%%%%%%%%%%%%%%%%%%%%%%%%%%%%%%%%%%%%%%%%

%%%%%%%%%%%%%%%%%%%%%%%%%%%%%%%%%%%%%%%%%%%%%%%%%%
\begin{proof}[Proof of Lemma \ref{lem-3.1}]
Take
$\psi:=v^-{\bf 1}_{(-K,K)\times[0,t_0]}$
as a test function, where $v^-:=\max\{-v,0\}$
and we also used ${\bf 1_\Omega}$ as  the characteristic function of $\Omega\subset \R^2$.
Since
$(v^-)_t=-v_t{\bf 1}_{ \{v(\cdot,t)<0\}}$,
we have
\beaa
\int_0^T\int_\bR v\psi_tdxdt=-\frac12 \int_{-K}^K v^2(x,t_0)\cdot{\bf 1}_{\{ v(\cdot,t_0)<0\}}dx.
\eeaa
Therefore, for any large $K$ and any $t_0\in(0,T)$, it follows from \eqref{eq-w2} that $v^-(x,t_0)=0$ for all $x\in(-K,K)$.
\end{proof}
%%%%%%%%%%%%%%%%%%%%%%%%%%%%%%%%%%%%%%%%%%%%%%%%%%

%%%%%%%%%%%%%%%%%%%%%%%%%%%%%%%%%%%%%%%%%%%%%%%%%%
\begin{proof}[Proof of Lemma \ref{lem:wk-t}]
For any given $\tau\in(0,T)$, we replace $\varphi$ in \eqref{eq-w1} by $\varphi\eta_{\varepsilon}(t)$,
where
\beaa
\eta_{\varepsilon}(t):=
                         \left\{\begin{array}{ll}
                         1,\quad &0\leq t\leq\tau-\varepsilon,\vspace{1mm}\\
                         \dfrac {\tau-t}{\varepsilon},\quad &\tau-\varepsilon\leq t\leq \tau,\vspace{1mm}\\
                         0,\quad &\tau \leq t \leq T.
                         \end{array}\right.
\eeaa
It follows from \eqref{eq-w1} that
\bea\label{eq-w1-tau}
&&-\int_{\R}{\bf 1}_{\Omega(0)}\varphi(x,0) dx=
\int_0^T \int_{\Omega(t)}(\varphi\eta_{\varepsilon})_tdxdt+
\int_{\partial \Omega\cap (\R\times(0,T))}W(v)\varphi\eta_{\varepsilon} |n_1|d\sigma.
\eea
Note that
\beaa
\int_0^T \int_{\Omega(t)}(\varphi\eta_{\varepsilon})_tdxdt&=&\int_0^T \int_{\Omega(t)}\varphi_t\eta_{\varepsilon}dxdt
-\frac{1}{\varepsilon}\int_{\tau-\varepsilon}^\tau \int_{\Omega(t)}\varphi dxdt\\
&&\longrightarrow \int_0^\tau \int_{\Omega(t)}\varphi_t dxdt-\int_{\mathbb{R}}{\bf 1}_{\Omega(\tau)}\varphi(x,\tau) dx\quad \mbox{as $\varepsilon\to 0$}.
\eeaa
Therefore, taking $\varepsilon\to 0$ in \eqref{eq-w1-tau}, we obtain
\beaa
&&\left[\int_{\R}{\bf 1}_{\Omega(t)}\varphi dx\right]_0^\tau
=
\int_0^\tau \int_{\Omega(t)}\varphi_tdxdt+
\int_{\partial \Omega\cap (\R\times(0,\tau))}W(v)\varphi |n_1|d\sigma.
\eeaa

Similarly, if we replace $\psi$ in \eqref{eq-w2} by $\psi\eta_{\varepsilon}(t)$ with $\eta_{\varepsilon}(t)$ defined above,
then the above process can be applied to obtain
\beaa
\left[\int_{\R}v\psi dx\right]_0^\tau=
\int_0^\tau\int_{\R}\Big(v\psi_t+g({\bf 1}_{\Omega(t)},v)\psi\Big )dxdt.
\eeaa
Hence the proof of Lemma~\ref{lem:wk-t} is completed.
\end{proof}
%%%%%%%%%%%%%%%%%%%%%%%%%555
%%%%%%%%%%%%%%%%%%%%%%%%%%%%%%%%%%%%%%%%%%%%%%%%%%
\begin{proof}[Proof of Lemma~\ref{lem:extend}]
By the assumption, we have
\beaa
&&\left[\int_{\R}{\bf 1}_{\Omega_1(t)}\varphi dx\right]_0^{T_1}
={\int_0^{T_1} \int_{\Omega_1(t)}\varphi_tdxdt+}\int_{\partial \Omega_1\cap (\R\times(0,T))}W(v_1)\varphi |n_1|d\sigma,\\
&&\left[\int_{\R}v_1\psi dx\right]_0^{T_1}
=\int_0^{T_1}\int_{\R}\Big(v_1\psi_t+g({\bf 1}_{\Omega_1(t)},v_1)\psi\Big )dxdt,\\
&&\left[\int_{\R}{\bf 1}_{\Omega_2(t)}\varphi dx\right]_{T_1}^{T_2}
={\int_{T_1}^{T_2} \int_{\Omega_2(t)}\varphi_tdxdt+}\int_{\partial \Omega_2\cap (\R\times(0,T))}W(v_2)\varphi |n_1|d\sigma,\\%\label{eq-w1}\\
&&\left[\int_{\R}v_2\psi dx\right]_{T_1}^{T_2}
=\int_{T_1}^{T_2}\int_{\R}\Big(v_2\psi_t+g({\bf 1}_{\Omega_2(t)},v_2)\psi\Big )dxdt.
\eeaa
Let $(\widehat{\Omega},\widehat{v})$ be as in \eqref{ext-omega} and \eqref{ext-v}. By adding the above equalities,
we obtain \eqref{eq-w1} and \eqref{eq-w2} for $0\leq t\leq T_2$ and for $(\widehat{\Omega},\widehat{v})$.
By the assumptions, it is not hard to obtain
$(\widehat{\Omega},\widehat{v})\in X_{T_2}$ and satisfies {\bf(C2)}.
Hence $(\widehat{\Omega},\widehat{v})$ is a weak solution for $0\leq t\leq T_2$ and then the proof is completed.
\end{proof}
%%%%%%%%%%%%%%%%%%%%%%%%%%%%%%%%%%%%%%%%%%%%%%%%%%

%%%%%%%%%%%%%%%%%%%%%%%%%%%%%%%%%%%%%%%%%%%%%%%%%%%%%%%%%

%%%%%%%%%%%%%%%%%%%%%%%%%%%%%%%%%%%%%%%%%%%%%%%%%%%%%%%%%%%5

\end{document}